\documentclass[11pt]{article}   % you have 10pt, 11pt, or 12pt options

\usepackage{amssymb}
\usepackage{amsfonts}
\usepackage{amsmath,bm}
\usepackage{latexsym}
\usepackage{amsthm}
\usepackage{enumerate}
\usepackage{epsfig}
\usepackage{graphicx}
\usepackage{color}
\usepackage{float}
\usepackage{subfigure}
\usepackage{amsmath,enumerate}
\usepackage{makeidx}
\usepackage{fancyhdr}
\usepackage{lastpage}
\usepackage{url}
\usepackage{faktor}
\newcommand{\Mod}[1]{\ (\text{mod}\ #1)}

\definecolor{darkblue}{rgb}{0,0,0.8}

\DeclareSymbolFont{AMSb}{U}{msb}{m}{n}  
\DeclareMathSymbol{\Sph}{\mathbin}{AMSb}{"53} \DeclareMathSymbol{\R}{\mathbin}{AMSb}{"52}
\DeclareMathSymbol{\T}{\mathbin}{AMSb}{"54} \DeclareMathSymbol{\Z}{\mathbin}{AMSb}{"5A}
\DeclareMathSymbol{\K}{\mathbin}{AMSb}{"4B}

\newtheorem{example}{Example}[section]
\newtheorem{guess}[example]{Theorem}
\newtheorem{defin}[example]{Definition}
\newtheorem{coro}[example]{Corollary}
\newtheorem{llemma}[example]{Lemma}

\newcommand{\Aut}{\mathrm{Aut}}
\newcommand{\Hol}{\mathrm{Hol}}
\newcommand{\Cay}{\mathrm{Cay}}

\newcounter{case}

\begin{document}  % necessary part of document

\title{The NNN-Property of Cyclic Groups}

\author{Michael Giudici$^{1}$\footnote{Email: michael.giudici@uwa.edu.au}, 
Luke Morgan$^{2}$\footnote{Email: luke.morgan@famnit.upr.si}, 
Yian Xu$^{3}$\footnote{Email: yian$\_$xu@seu.edu.cn}\\ 
 $^{1}$ Centre for the Mathematics of Symmetry and Computation \\
 Department of Mathematics and Statistics  \\
  The University of Western Australia \\
  35 Stirling Highway, Crawley, WA, 6009, Australia \\
  $^{2}$ University of Primorska, UP FAMNIT, Glagolja\v{s}ka 8, 6000 Koper, Slovenia\\
  University of Primorska, UP IAM, Muzejski trg 2, 6000 Koper, Slovenia \\
  $^{3}$ School of Mathematics  \\
  Southeast University\\
  2 SEU Road, Nanjing, 211189, China
}

\date{ }
\maketitle

\begin{abstract}
A Cayley graph is said to be an NNN-graph if it is both normal and non-normal for isomorphic regular groups, and a group has the NNN-property if there exists an NNN-graph for it. In this paper we investigate the NNN-property of cyclic groups, and show that cyclic groups do not have the NNN-property.  
\end{abstract}


\section{Introduction}

Let $G$ be a finite group, and $S$ be a subset of $G$ such that $S$ does not contain the identity of $G$ and $S=S^{-1}=\{s^{-1}|s\in S\}$. The $\it{Cayley\ graph}$ $\Gamma=\Cay(G, S)$ is defined to have vertex set $V(\Gamma)=G$, and edge set $E(\Gamma)=\{\{g, sg\}|s\in S\}$. We say that $G$ is the {\it defining group} of $\Gamma$ and $S$ is the {\it connection set} of $\Gamma$. Let $\Aut(\Gamma)$ denote the automorphism group of $\Gamma$. For each $g\in G$, define a map $g_{R}: G\to G$ by the right multiplication of $g$ on $G$ as below: 
$$
g_{R}: x\to xg, \ \text{for}\ x\in G.
$$
Then $g_{R}$ is an automorphism of $\Gamma$. It follows from the definition that the group $G_{R}=\{g_{R} \mid g\in G\}$ is a subgroup of $\Aut(\Gamma)$ and acts regularly on $V(\Gamma)$. It is well known that a graph is a Cayley graph if and only if its automorphism group contains a subgroup that acts regularly on the vertex set of the graph (see \cite{sabidussi1958class}, \cite[Theorem 16.3]{MR1271140}). 

We say that $\Gamma$ is a $\it{normal\ Cayley\ graph\ for\ G}$  (Xu, \cite{Xu1998}) (or normal) if $G_{R}\unlhd \Aut(\Gamma)$, otherwise we say that $\Gamma$ is a $\it{non}$-$\it{normal\ Cayley\ graph\ for\ G}$ (or non-normal). If $\Gamma$ is normal for $G$, then $\Aut(\Gamma)=G_{R}\rtimes \Aut(G, S)$. Xu \cite{Xu1998} showed that, except for $Z_{4}\times Z_{2}$ and $Q_{8}\times Z_{2}^{m}$ with $m\geq 0$, each finite group has at least one normal Cayley graph. (We use $Z_n$ to denote a cyclic group of order $n$.) Years later, Feng and Dobson (see \cite{bamberg2011point}) proposed a question asking if it is possible for a Cayley graph to be both normal and non-normal for two isomorphic regular groups, that is, can $\Aut(\Gamma)$ contain a normal regular subgroup $G$ and a non-normal regular subgroup isomorphic to $G$.

\begin{defin}\label{defin11}
A Cayley graph $\Gamma=\mathrm{Cay}(G, S)$ is an {\emph{NNN-graph for $G$}} (or \emph{NNN for $G$}) if $\Gamma$ is normal for $G$ and is non-normal for $H$ where $H\cong G$ and $H\neq G$. 
\end{defin}

\begin{defin}\label{defin12}
A group $G$ has the \emph{NNN-property} (or $G$ is an \emph{NNN-group}) if there exists an NNN-graph for $G$.
\end{defin}

Few NNN-graphs are known. Giudici and Smith \cite{giudici2010note} constructed a strongly regular Cayley graph for $Z_{6}^{2}$ and showed that such a graph is an NNN-graph for $Z_{6}^{2}$. Royle \cite{royle2008normal} proved that the halved folded $8$-cube is an NNN-graph for $Z_{2}^{6}$. The first infinite family of NNN-graphs was found by Bamberg and Giudici \cite{bamberg2011point} when they studied the point graphs of a particular family of generalised quadrangles. It was shown in \cite{MR3698085} that  
the Cartesian product, direct product or strong product  of an NNN-graph  with finitely many normal Cayley graphs  is an NNN-graph.
%applying Cartesian product, direct product and strong product  to an NNN-graph and finitely many normal Cayley graphs  produces more NNN-graphs. 
Taking such products of the
%By applying this method to the
 halved folded 8-cube with
  finitely many $K_{2}$, it is proved in \cite{yxuthesis2019} that the elementary abelian 2-group $Z_{2}^{d}$ has NNN-graphs if $d\geqslant 6$.

A graph is a {\it CI-graph} for some finite group $G$ if all regular subgroups of automorphisms isomorphic to $G$ are conjugate, and so any CI-graph is not NNN. A group $G$ is said to be a {\it CI-group} if all Cayley graphs for $G$ are CI-graphs. Clearly, 
%every
a CI-group does not have the NNN-property, and so the study of CI-groups can help us investigate the conditions on %of
the existence of NNN-groups. For example, $Z_{p}^{d}$ is a CI-group for an arbitrary prime $p$ if $d\leqslant5$ (see \cite{yxuthesis2019}). Since there exist NNN-graphs for $Z_{2}^{d}$ when $d\geqslant 6$, we have that $Z_{2}^{d}$ does not have the NNN-property if and only if $d\leqslant 5$. It is shown in \cite{Muzychuk1997} that if $n\in\{8,9,18\}$ or $n=k,2k$ or $4k$ where $k$
is an odd square-free integer,
% is odd square-free,
  then the cyclic group $Z_{n}$ is a CI-group, and so it does not have the NNN-property. This motivates us to determine if arbitrary cyclic groups have the NNN-property. 

A \emph{circulant}  is a Cayley graph of a cyclic group. The essence of testing whether a circulant $\Gamma$ of a group $G$ is NNN for $G$ is to study the cyclic regular subgroups contained in $\mathrm{Aut}(\Gamma)$. Even though the investigation of the automorphism groups and the regular subgroups of circulants has a long history (see \cite{Alspach1979, Fernandes2005}), most related work considered non-isomorphic regular subgroups. Joseph \cite{joseph1995isomorphism} showed that if a graph is a Cayley graph of a cyclic group and a non-cyclic group of order $p^{2}$ for $p$ a prime, then the graph is a lexicographic product of two Cayley digraphs of prime order. Also, Maru\v{s}i\v{c} and Morris \cite{maruvsivc2005normal} gave a sufficient but not necessary condition for a graph 
to be
%being
 a Cayley graph for a cyclic group and a non-cyclic group. 

In \cite{morris1999isomorphic}, Morris extended Joseph's result to Cayley graphs of abelian groups of odd prime-power order, and proved that a digraph is a Cayley digraph of a cyclic group and a non-isomorphic abelian group if and only if it is a wreath product of Cayley graphs of $p$-groups.  Kov\'{a}cs and Servatius \cite{kovacs2012cayley} extended Morris's result
to the case when $p=2$
% forthe case when the prime equals $2$
   and gave a necessary and sufficient condition for two Cayley digraphs 
to be  % being
    isomorphic. One of the ideas %techniques 
    used
     in their arguments is 
     %introducing
      the notion of {\it W-subgroups of abelian groups}: a subgroup $H\leqslant G$ is a \emph{W-subgroup relative to $S$} where $S\subseteq G\backslash \{\bf 1\}$, if $S\setminus H$ is a union of $H$-cosets, denoted by $H\leqslant_{S} G$. In particular, if $H\leqslant_{S}G$ and $H\neq G$, then we write $H<_{S}G$. Kov\'{a}cs and Servatius showed that if $G$ has a W-subgroup $H<_{S}G$, then the graph $\Gamma=\Cay(G, S)$ is a lexicographic product of two nontrivial circulants (see \cite[Theorem 1.2]{kovacs2012cayley}), and it follows from \cite[Lemma 5.3.4]{yxuthesis2019} that $\Gamma$ is not NNN for $G$. Hence, the study of W-subgroups of $Z_{2^{n}}$ can help us determine some cases where  NNN-graphs do not exist. However, Kov\'{a}cs and Servatius proved that if $Z_{2^{n}}$ has no $H<_{S}Z_{2^{n}}$, then every abelian regular subgroup of $\Aut(\Cay(Z_{2^{n}}, S))$ is cyclic (see \cite[Lemma 4.5]{kovacs2012cayley}), which implies that considering the W-subgroups of $Z_{2^{n}}$ is not enough 
to determine the NNN-property of cyclic groups.     %for solving       the NNN-property for cyclic groups.
 Recently, by studying arc-transitive circulants, Li, Xia and Zhou \cite{li2019explicit} showed that an arc-transitive circulant is normal if and only if its automorphism group contains a unique regular cyclic subgroup, which implies that arc-transitive normal circulants are not NNN-graphs. 

The aim of this paper is to prove the following result.

\begin{guess}\label{main0}
Cyclic groups do not have the NNN-property.
\end{guess}

The proof of Theorem~\ref{main0} is structured as 
follows. %below. 
We first study the Sylow $p$-subgroups contained in the automorphism group of a normal circulant, and show that the automorphism group of normal circulants whose order is not divisible by 8 contains a unique abelian regular subgroup, and so cyclic groups with order not divisible by 8 do not have the NNN-property.

For cyclic groups with order divisible by 8, we
determine % solve
  the NNN-property by analysing the regular subgroups of their holomorphs. The {\it holomorph} of a group $G$ is $\mathrm{Hol}(G)=G_{R}\rtimes \mathrm{Aut}(G)$. We view $\Hol(G)$ as a permutation group on $G$. If $\Gamma$ is normal for $G$, then $\Aut(\Gamma)\leqslant \Hol(G)$, and so each regular subgroup of $\Aut(\Gamma)$ is a regular subgroup of $\Hol(G)$.  Hence if the only regular subgroups of $\Hol(G)$ that are isomorphic to $G$ are normal in $\Hol(G)$ then $G$ does not have the NNN-property. For example, by \cite[Theorem 7.7]{MR3807043} and \cite[Theorem 4]{MR1704676}, we have that the the only regular subgroups of the  holomorph of a nonabelian simple group $T$ that are isomorphic to $T$ are normal in $\Hol(T)$,  which implies that nonabelian simple groups do not have the NNN-property. The investigation of regular subgroups of the holomorph of a group is closely connected with Hopf-Galois structures as well as set-theoretic solutions of the Yang-Baxter equation (see \cite{MR3647970}). General descriptions of regular subgroups of the holomorph of finite groups are given in \cite{MR3807043, tsang2019multiple}, and regular subgroups of the  holomorph of small groups are calculated in \cite{MR3647970}. 

It is shown \cite{tsang2019multiple} that for an abelian group $G$ of order not divisible by 8, $\Hol(G)$ has a unique regular subgroup isomorphic to $G$. Hence such groups do not have the NNN-property (we have already observed this when $G$ is cyclic).  Moreover, Kohl \cite{MR1644203} showed that if $p$ is odd then any regular subgroup of $\Hol(Z_{p^{n}})$ is cyclic. However, things are very different for $p=2$. In Section 3,  we classify the regular subgroups of $\Hol(Z_{2^{n}})$, which is crucial for solving the NNN-property of cyclic groups whose order is divisible by 8.  Before stating the result we need to introduce some notation. Note that $\Aut(G)=\langle x\rangle \times \langle y\rangle$ where $x: a\to a^{-1}$ and $y: a\to a^{5}$. Up to isomorphism, there are three nonabelian groups of order $2^{n}$ with a cyclic subgroup $\langle \sigma\rangle$ of index two and an involution $\tau\notin\langle\sigma\rangle$ (see \cite[Chapter 5, Exercise 17]{MR2286236}), and they are as follows:
\begin{enumerate}
\item the dihedral group $D_{2^n}$, where $\sigma^{\tau}=\sigma^{-1}$;
\item the quasidihedral group $QD_{2^n}$, where $\sigma^{\tau}=\sigma^{2^{n-1}-1}$;
\item $M_{n}(2)$, where $\sigma^{\tau}=\sigma^{2^{n-1}+1}$. 
\end{enumerate}
We also encounter the 
generalised quaternion group %dihedral group
 $Q_{2^n}=\langle  \sigma, \tau\mid \sigma^{2^{n-1}}=\tau^4=1, \sigma^{2^{n-2}}=\tau^2, \sigma^\tau=\sigma^{-1}   \rangle$ of order $2^n$.

\begin{guess}\label{ch6maintheo}
Let $G=\langle a\rangle$ be a cyclic group that is isomorphic to $Z_{2^{n}}$ with $n\geqslant 3$, and $H=\Hol(G)=G_R\rtimes\Aut(G)$ be the holomorph of $G$. Suppose that $R$ is a subgroup of $H$. Then $R$ is regular if and only if, up to conjugacy, exactly one of the following holds:
\begin{enumerate}
\item $R=G_R$;
\item $R=\langle ay^{2^{t}}\rangle\cong Z_{2^{n}}$ for $0\leqslant  t\leqslant  n-3$ and $R\cap G_R=\langle a^{2^{n-t-2}}\rangle$.
\item $R=\langle a^{2}, ax\rangle\cong D_{2^{n}}$;
\item $R=\langle a^{2}, axy^{2^{n-3}}\rangle\cong Q_{2^{n}}$;
\item $R=\langle a^{2\cdot 5^{-1}}y\rangle\times \langle ax\rangle$ and $R\cap G_R=\langle a^{2^{n-1}}\rangle$;
\item $R=\langle a^{2\cdot 5^{-1}+2^{n-2}}y\rangle \rtimes \langle ax\rangle$ with $R\cap G_R=\langle a^{2^{n-1}}\rangle$, and $R\cong QD_{2^{n}}$;
\item $R=\langle a^{2}y^{2^{n-3}}\rangle\rtimes \langle ax\rangle$ with $R\cap G_R=\langle a^{4}\rangle$, and $R\cong M_{n}(2)$.
\end{enumerate} 
\end{guess}

In cases 5~and 6, we write $5^{-1}$ for the inverse of $5$ in the multiplicative group of $Z_{2^n}$.

\section{Cyclic Groups $G$ with $8\nmid |G|$}

We start with a result about centralisers of subgroups.

\begin{llemma}\label{ch5lem1}
Suppose that  $G=Z_{p^{k}}$ where $p$ is a prime and $k\geqslant 2$. Let $N=Z_{p^{m}}$ be a nontrivial proper subgroup of $G$, and $H=C_{\mathrm{Aut}(G)}(N)$. Then $|H|=p^{k-m}$. Moreover, $H\cong Z_{p^{k-m}}$, unless $p=2$ and $m=1$, in which case $H=\Aut(G)$. 
\end{llemma}

\begin{proof} Let $G=\langle a\rangle$. Then we have that $N=\langle a^{p^{k-m}}\rangle$. Since $N$ is a characteristic subgroup of $G$, there exists a homomorphism $\pi: \mathrm{Aut}(G)\mapsto\mathrm{Aut}(N)$ where $\mathrm{Ker}(\pi)=H=C_{\mathrm{Aut}(G)}(N)$. Clearly $\pi$ is surjective, and so 
\begin{equation}\label{ch5equlem1}
|H|=\frac{|\mathrm{Aut}(G)|}{|\mathrm{Aut}(N)|}=p^{k-m}.
\end{equation}
If $p$ is an odd prime, then (\ref{ch5equlem1}) and the fact that $\Aut(G)$ is cyclic imply that $H\cong Z_{p^{k-m}}$. 

Suppose that $p=2$. Here $\Aut(G)=\langle x\rangle \times \langle y\rangle\cong Z_2\times Z_{2^{k-2}}$, where $x:a\mapsto a^{-1}$ and $y:a\mapsto a^5$. If $m=1$, then $|\Aut(N)|=1$, in which case $H=\mathrm{Aut} (G)$. Now suppose that $m\geqslant 2$. Since $H\leqslant \Aut(G)$, %by Remark~\ref{ch1rem01} 
    we first suppose that $xy^{i}\in H$ for some $1\leqslant i\leqslant p^{k-2}$, that is, $(xy^{i})^{\pi}={\bf1}$. Thus  $x^{\pi}=(y^{i})^{\pi}=(y^{\pi})^{i}$. Thus Im$(\pi)=\Aut(N)=\langle x^{\pi}, y^{\pi}\rangle=\langle y^{\pi}\rangle$, which implies that $m=2$ as $\Aut(N)$ is cyclic if and only if $N=Z_{2^{2}}$ and $\Aut(N)\cong Z_{2}$. Thus $\langle y^{\pi}\rangle\cong Z_{2}$. Note that $N=\langle a^{2^{k-2}}\rangle=Z_{2^{2}}$. However, %by Remark~\ref{ch1rem01} 
    we have that $(a^{2^{k-2}})^{y}=a^{2^{k-2}\cdot 5}=a^{2^{k-2}(2^{2}+1)}=a^{2^{k-2}}$. Thus $y\in$ Ker$(\pi)$,  contradicting $\langle y^{\pi}\rangle\cong Z_{2}$. Hence we conclude that $xy^{i}\notin H$ for all $1\leqslant i\leqslant p^{k-2}$. Thus either $H\leqslant \langle y\rangle$ or $H=\langle x\rangle$. In each of these two cases, we have that $H$ must be cyclic. Hence $H\cong Z_{2^{k-m}}$. \end{proof}

Let $G=Z_{n}$. Note that $n$ has a prime factorization $n=p_{1}^{k_{1}}\cdots p_{t}^{k_{t}}$ where $p_{i}$ is a prime. If $n$ is even, then we set $p_{1}=2$. Thus  
\begin{equation}\label{ch5equ1}
G\cong Z_{p_{1}^{k_{1}}}\times\cdots\times Z_{p_{t}^{k_{t}}},
\end{equation}
and 
\begin{equation}\label{ch5equ2}
\mathrm{Aut} (G)=\Aut(Z_{p_{1}^{k_{1}}})\times \cdots\times \Aut(Z_{p_{t}^{k_{t}}}).
\end{equation}
Let $a=(a_{1}, \ldots, a_{t})$ be a generator of $G$ where $a_{i}$ generates $Z_{p_{i}^{k_{i}}}$. It follows easily from (\ref{ch5equ1}) and (\ref{ch5equ2}) that for $x=(x_{1}, \ldots, x_{t})\in G$ and $\pi=(\pi_{1}, \ldots, \pi_{t})\in \mathrm{Aut} (G)$ we have that, 
\begin{equation}\label{ch5equ5}
x^{\pi}=(x_{1}^{\pi_{1}}, \ldots, x_{t}^{\pi_{t}}). 
\end{equation}
Let $\Gamma=\mathrm{Cay}(G, S)$ be a circulant and $A=\mathrm{Aut}(\Gamma)$.

\begin{llemma}\label{ch5lemcentralizer}
Suppose that $H\leqslant \Hol(G)$ is an abelian regular subgroup such that $H\neq G_{R}$. Let $N=G_{R}\cap H$. Then $N=Z_{p_{1}^{m_{1}}}\times \cdots\times Z_{p_{t}^{m_{t}}}$ 
with $1\leqslant m_{i}\leqslant k_{i}$ for all $1\leqslant i\leqslant t$, and  
\begin{equation}\label{ch5equ6}
C_{\mathrm{Aut} (G)}(N)=\prod_{i=1}^{t}C_{\Aut(Z_{p_{i}^{k_{i}}})}(Z_{p_{i}^{m_{i}}}).  
\end{equation}
\end{llemma}

\begin{proof} Recall that $\Hol(G)=G_{R}\rtimes \mathrm{Aut}(G)$. By the Second Isomorphism Theorem, we have that 
\begin{equation*}
H/N\cong HG_{R}/G_{R}\leqslant \Hol(G)/G_{R}=\mathrm{Aut}(G).
\end{equation*} 
Hence
\begin{equation}\label{ch5equ10}
\frac{|H|}{|N|} \ \quad \textrm{divides}\quad \ p_{1}^{k_{1}-1}(p_{1}-1)\cdots p_{t}^{k_{t}-1}(p_{t}-1).
\end{equation}
Since the right hand side is less than $|H|$, we have that $|N|\neq 1$. Thus $N=\prod_{i\in I}Z_{p_{i}^{m_{i}}}$ with $1\leqslant m_{i}\leqslant k_{i}$ and $\emptyset \neq I\subseteq \{1, \ldots, t\}$. Here notice that $N\neq H$ as we assume $H\neq G_{R}$.  

Let $C_{\Hol(G)}(N)$ be the centraliser of $N$ in $\Hol(G)$. Obviously $G_{R}\leqslant C_{\Hol(G)}(N)$. Thus $C_{\Hol(G)}(N)=G_{R}\rtimes C_{\mathrm{Aut} (G)}(N)$. Moreover, $H\leqslant C_{\Hol(G)}(N)$ as $H$ is abelian. Again by the Second Isomorphism Theorem, 
\begin{equation}\label{ch5equ4}
H/N\cong HG_{R}/G_{R}\leqslant C_{\mathrm{Aut} (G)}(N).
\end{equation}

Let $\sigma\in C_{\mathrm{Aut} (G)}(N)$. Using (\ref{ch5equ1}) and (\ref{ch5equ2}), write $n=(n_{1} \ldots, n_{t})$ and $\sigma=(\sigma_{1}, \ldots, \sigma_{t})$ where $\sigma_{i}\in \Aut(Z_{p_{i}^{k_{i}}})$ and $n_{i}\in Z_{p_{i}^{m_{i}}}$ when $i\in I$, and $n_{j}={\bf 1}$ for $j\notin I$. Then by (\ref{ch5equ5}) we have that 
\begin{equation*}
n^{\sigma}=(n_{1}^{\sigma_{1}},  \ldots, n_{t}^{\sigma_{t}})
\end{equation*}
Thus an automorphism $\sigma$ of $G$ centralises $N$ if and only if  each $\sigma_{i}$ fixes each element in $Z_{p_{i}^{m_{i}}}$ for $i\in I$. By (\ref{ch5equ2}), we have that
\begin{equation}\label{ch5equcentralizer}
C_{\mathrm{Aut} (G)}(N)=\left(\prod_{i\in I}C_{\Aut(Z_{p_{i}^{k_{i}}})}(Z_{p_{i}^{m_{i}}})\right)\left(\prod_{j\notin I}\Aut(Z_{p_{j}^{k_{j}}})\right).
\end{equation}
It follows from (\ref{ch5equ4}) %, Proposition~\ref{ch1pro02} 
and Lemma~\ref{ch5lem1} that 
\begin{equation}\label{ch5equ7}
|H/N|=\left(\prod_{i\in I}p_{i}^{k_{i}-m_{i}}\right)\left(\prod_{j\notin I}p_{j}^{k_{j}}\right) \quad \textrm{divides}\quad  \left(\prod_{i\in I}p_{i}^{k_{i}-m_{i}}\right)\left(\prod_{j\notin I}p_{j}^{k_{j-1}}(p_{j}-1)\right).
\end{equation}

Suppose that  $I$ is a proper subset of $\{0,1 , \ldots, t\}$. Let $p_{w}$ be the largest prime such that $w\notin I$. Then $p_{w}^{k_{w}}$ divides $ |H/N|$, and so by (\ref{ch5equ7}) we have that $p_{w}^{k_{w}}$ divides $\prod_{j\notin I}p_{j}^{k_{j-1}}(p_{j}-1)$, which leads to a contradiction. Thus $w\in I$. This implies that $I=\{1, \ldots, t\}$, and so $N=Z_{p_{1}^{m_{1}}}\times \cdots \times Z_{p_{t}^{m_{t}}}$ 
with $1\leqslant m_{i}\leqslant k_{i}$ for all $1\leqslant i\leqslant t$, and so (\ref{ch5equ6}) follows. \end{proof}

The following corollary follows from Lemma~\ref{ch5lem1} and Lemma~\ref{ch5lemcentralizer}. 

\begin{coro}\label{ch5corocentralizer}
Let $H\leqslant \Hol(G)$ be an abelian regular subgroup. Let $N=G_{R}\cap H$. 
\begin{enumerate}[(1)]
\item If $n$ is odd, then
\begin{equation*}
C_{\mathrm{Aut} (G)}(N)=\prod_{i=1}^{t}Z_{p_{i}^{k_{i}-m_{i}}},
\end{equation*} 
\item If $n$ is even, then 
\begin{equation*}%\label{ch5equ61}
C_{\mathrm{Aut} (G)}(N) =\begin{cases}
               \Aut(Z_{2^{k_{1}}})\times \prod_{i=2}^{t}Z_{p_{i}^{k_{i}-m_{i}}}      & \mathrm{if}\  m_{1}=1, \\
               Z_{2^{k_{1}-m_{1}}}\times \prod_{i=2}^{t}Z_{p_{i}^{k_{i}-m_{i}}}   & \mathrm{if} \ 2\leqslant m_{1}\leqslant k_{1}. 
               \end{cases}
\end{equation*}
\end{enumerate}
Moreover, $|C_{\mathrm{Aut} (G)}(N) |=|H: N|$. 
\end{coro}

\begin{llemma}\label{ch5lem2}
Let $\Phi=(\phi_{1}, {\bf 1}, \ldots, {\bf 1})\in \mathrm{Aut} (G)$ with $|\phi_{1}|=p_{1}$. Suppose that $p_{1}$ is odd. If $\Gamma$ is a normal circulant, then $\Phi\notin \mathrm{Aut} (G, S)$.
\end{llemma}

\begin{proof} {\it (The idea for the arguments of this lemma arises from the proof of \cite[Theorem 3.2]{xu2004})}. Recall that $G=\langle a\rangle$ with $a=(a_{1}, \ldots, a_{t})$. Let $M=p_{2}^{k_{2}}\cdots p_{t}^{k_{t}}$. We know that there is a unique subgroup $P$ of order $p_{1}$ in $G$, namely, 
$$P=\langle a^{Mp_{1}^{k_{1}-1}}\rangle.$$

Let $S_{1}=\{s\in S\mid s=a^{r}\ \mathrm{with} \ (r, p_{1})=1\}$, $S_{3}=S\cap (P\backslash\{\bf 1\})$, and $S_{2}=S\backslash (S_{1}\cup S_{3})$. Let $X_{1}=\mathrm{Cay} (G, S_{1}\cup S_{3})$ and $X_{2}=\mathrm{Cay} (G, S_{2})$. Thus $E(\Gamma)=E(X_{1})\cup E(X_{2})$. 

Suppose to the contrary that $\Phi\in \mathrm{Aut} (G, S)$. Note that this implies $k_{1}> 1$. Then for each $s\in S$, we have that $s^{\langle\Phi\rangle}\subseteq S$. Since $|\phi_{1}|=p_{1}$, we may assume that  
\begin{equation*}
\Phi: (a_{1}, \ldots, a_{t})\to (a_{1}^{p_{1}^{k_{1}-1}+1}, a_{2}, \ldots, a_{t}).
\end{equation*}

Let $s\in S_{1}$, and suppose that  $s=a^{r}$. Then 
\begin{equation*}
s^{\Phi}=(a_{1}^{r}, \ldots, a_{t}^{r})^{\Phi}=(a_{1}^{r(p_{1}^{k_{1}-1}+1)}, a_{2}^{r}, \ldots, a_{t}^{r})=s(a_{1}^{rp_{1}^{k_{1}-1}}, {\bf 1}, \ldots, {\bf 1})=sa^{Mrp_{1}^{k_{1}-1}}.
\end{equation*}
Since $\Phi$ centralises $P$, we have that 
\begin{equation*}
s^{\langle\Phi\rangle}=\{sa^{Mrp_{1}^{k_{1}-1}}, \ldots, sa^{M(p_{1}-1)rp_{1}^{k_{1}-1}}\}.
\end{equation*}
Since $(r, p_{1})=1$, we have that $s^{\langle\Phi\rangle}=sP$.

Since $p_{1}$ is odd, there is a permutation $\theta: G\to G$ defined as below: for each $x\in G$, write $x=a^{r_{x}}$ and then define 
\begin{equation}\label{ch5equ9}
x^{\theta}= \begin{cases}
               x               & \mathrm{if} \ r_{x}\not\equiv 2\Mod {p_{1}}, \\
               xa^{Mp_{1}^{k_{1}-1}}               & \mathrm{if} \ r_{x} \equiv 2\Mod {p_{1}}.
               \end{cases}
\end{equation}
We claim that $\theta\in A_{\bf 1}$ but $\theta\notin \mathrm{Aut} (G, S)$. Let $u=a^{r_{u}}$ and $v=a^{r_{v}}$. Suppose that  $\{u, v\}\in E(\Gamma)$. Then $uv^{-1}=a^{r_{u}-r_{v}}=s$ for some $s\in S$. 

Suppose first that $r_{u}\not\equiv 2\Mod {p_{1}}$ and $r_{v}\not\equiv 2 \Mod {p_{1}}$. Thus $u^{\theta}(v^{\theta})^{-1}=uv^{-1}=s$, and so $\theta\in \Aut(\Gamma)$. Next suppose that $r_{u}\equiv 2\Mod {p_{1}}$ and $r_{v}\equiv 2 \Mod {p_{1}}$. Thus  
\begin{equation*}
u^{\theta}(v^{\theta})^{-1}=ua^{Mp_{1}^{k_{1}-1}}(a^{Mp_{1}^{k_{1}-1}})^{-1}v^{-1}=uv^{-1}=s,
\end{equation*}
and so $\theta\in \Aut(\Gamma)$.

Finally, swapping $u$ and $v$ if necessary, we are left to consider $r_{u}\equiv 2\Mod {p_{1}}$ and $r_{v}\not\equiv 2 \Mod {p_{1}}$. Notice that, if $s\in S_{2}\cup S_{3}$, we have that $p_{1}\mid (r_{u}-r_{v})$, that is, $r_{u}\equiv r_{v}\Mod {p_{1}}$. This leads to a contradiction, and so we may assume that $s\in S_{1}$. Since 
\begin{equation*}
u^{\theta}(v^{\theta})^{-1}=ua^{Mp_{1}^{k_{1}-1}}v^{-1}=sa^{Mp_{1}^{k_{1}-1}},
\end{equation*} 
we have that $u^{\theta}(v^{\theta})^{-1}\in sP=s^{\langle\Phi\rangle}\subseteq S$. Thus $\theta\in \Aut(\Gamma)$. Hence we have shown that $\theta\in \Aut(\Gamma)$ in all cases. 

In fact $\theta\in A_{\bf 1}$ as $\theta$ fixes {\bf 1}. Also since $\theta$ fixes the group generator $g$, if $\theta\in \mathrm{Aut} (G)$, then $\theta={\bf 1}$, which is a contradiction by the definition of $\theta$. Thus $\theta\notin \mathrm{Aut} (G, S)$, and so $\Gamma$ is non-normal, which is a contradiction. Therefore, we must have $\Phi\notin \mathrm{Aut} (G, S)$. \end{proof}

\begin{coro}\label{ch5cor1}
For each $1\leqslant i \leqslant t$, let $\Phi_{i}=(\phi_{1}, \phi_{2}, \ldots, \phi_{t})$ where  $|\phi_i|=p_i$ with $p_i$ odd, and $\phi_j=1$ for all $j\neq i$.
 If $\Gamma$ is a normal circulant, then $\Phi_{i}\notin \mathrm{Aut} (G, S)$.
\end{coro}

\begin{proof} Since the order of the prime factors of $n$ was not specified in  (\ref{ch5equ1}) and (\ref{ch5equ2}), for each $1\leqslant i\leqslant t$ with $p_i$ odd, we may switch the $i^\mathrm{th}$ factor with the first factor. Then $\Phi_i$ becomes the $\Phi$ in Lemma \ref{ch5lem2} and the result follows.
\end{proof}

The next result is a consequence  following from Corollary~\ref{ch5corocentralizer} and Corollary~\ref{ch5cor1}.  

\begin{llemma}\label{ch5them31}
Let $\Gamma$ be a normal circulant for $G$ and $H\leqslant A=\Aut(\Gamma)$ be an abelian regular subgroup such that $H\neq G_{R}$. Let $N=G_{R}\cap H$. Then $|G_{R}:N|$ is a power of $2$.
\end{llemma}

%there exists no element in $G\backslash N$ that has an odd prime order.  

\begin{proof} Since $\Gamma$ is normal for $G$, we have that $A=G_{R}\rtimes \Aut(G, S)$. Let $C_{A}(N)$ be the centraliser of $N$ in $A$. Since $G_{R}\leqslant C_{A}(N)$, we have that $C_{A}(N)=G_{R}\rtimes C_{\mathrm{Aut} (G, S)}(N)$. Moreover, $H\leqslant C_{A}(N)$ as $H$ is abelian. Thus 
\begin{equation}\label{ch5equthem31}
H/N\cong HG_{R}/G_{R}\leqslant C_{\mathrm{Aut} (G, S)}(N)\leqslant C_{\mathrm{Aut} (G)}(N).
\end{equation}
It follows from Corollary~\ref{ch5corocentralizer} that $|C_{\Aut(G)}(N)|=|H:N|$. Thus by (\ref{ch5equthem31}) we have that $C_{\Aut(G, S)}(N)=C_{\Aut(G)}(N)$. Moreover, by Lemma~\ref{ch5lemcentralizer} we have that $N=\prod_{i=1}^{t}Z_{p_{i}^{m_{i}}}$ with $1\leqslant m_{i}\leqslant k_{i}$ for all $1\leqslant i\leqslant t$. 

Suppose that there exists some $1\leqslant j\leqslant t$ where $p_{j}$ is odd and $m_{j}\leqslant k_{j}-1$. Recall that $Z_{p_{j}^{k_{j}}}$ is generated by $a_{j}$. Since $k_{j}-m_{j}\geqslant 1$, there is an automorphism $\rho$ of order $p_{j}$ in $C_{\Aut(Z_{p_{j}^{k_{j}}})}(Z_{p_{j}^{m_{j}}})\cong Z_{p_{j}^{k_{j}-m_{j}}}$. Let $\tau=(\tau_{1}, \ldots, \tau_{t})\in C_{\mathrm{Aut}(G)}(N)$ where $\tau_{j}=\rho$ and $\tau_{i}={\bf 1}$ for all $i\neq j$. By Corollary~\ref{ch5cor1}, if $\Gamma$ is normal, then $\tau\notin \mathrm{Aut} (G, S)$, which is a contradiction to $C_{\mathrm{Aut} (G, S)}(N)=C_{\mathrm{Aut} (G)}(N)$. Thus for all odd prime factors $p_{i}$ where $1\leqslant i\leqslant t$, we have that $m_{i}=k_{i}$. Hence $|G_{R}:N|$ is a power of 2.  \end{proof}

\begin{guess}\label{ch5coroabelian}
Let $\Gamma$ be a normal circulant for $G$ with $4\nmid |G|$. Then $G_{R}$ is the unique abelian regular subgroup contained in $A$.  
\end{guess}

\begin{proof} First suppose that $n$ is odd. By Lemma~\ref{ch5them31} we have that $N=\prod_{i=1}^{t}Z_{p_{i}^{k_{i}}}$ which implies that $N=H=G_{R}$. 

Suppose that $2\mid n$. Since $4\nmid n$, we have that $G_{R}=Z_{2}\times \prod_{i=1}^{t}Z_{p_{i}^{k_{i}}}$. It follows from Lemma~\ref{ch5lemcentralizer} and Lemma~\ref{ch5them31} that $N=Z_{2}\times \prod_{i=1}^{t}Z_{p_{i}^{k_{i}}}$. Thus $N=H=G_{R}$.  \end{proof}

We note that Theorem \ref{ch5coroabelian} is a generalisation of \cite[Theorem 4.3]{maruvsivc2005normal} which shows that there are no non-cyclic regular abelian subgroups of $\Aut(\Gamma)$.

\begin{guess}\label{notdivisibleby8}
The cyclic group $Z_{n}$ with $8\nmid n$ does not have the NNN-property. 
\end{guess}

\begin{proof} By Theorem~\ref{ch5coroabelian} we are left to consider the cases where $n=4m$ and  $m$ is odd. Let $H$ be an abelian regular subgroup of $A$ such that $H\neq G_{R}$. It follows from Lemmas~~\ref{ch5lemcentralizer} and \ref{ch5them31} that $N=H\cap G_{R}=Z_{2}\times \prod_{i=1}^{t}Z_{p_{i}^{k_{i}}}$. Since $N$ is a characteristic subgroup of $G$, we have that $N\lhd \Hol(G)$. Thus 
$$
\Hol(G)/N\cong (G_{R}/N)\rtimes \Aut(G)\cong (G_{R}/N)\rtimes (Z_{2}\times \Aut(Z_{m}))\cong Z_{2}\times Z_{2}\times \Aut(Z_{m}). 
$$
Since the quotient group is abelian, we have that $H\lhd \Hol(G)$, and so $H\lhd A$. Hence $Z_{4m}$ does not have the NNN-property and the result follows. \end{proof}

Notice that if $p_{1}=2$, then there is no permutation as defined in (\ref{ch5equ9}) such that it is an automorphism of the graph. In fact, in general Lemma~\ref{ch5lem2} is not true for $G=Z_{2^{k}}$ with $k\geqslant 3$.  In this case $\Aut(G)$ has three different elements of order $2$, and there are normal circulants for $G$ such that $\Aut(G, S)$ contains at least three different involutions  (see \cite[pg.~65]{yxuthesis2019}).
 Hence the arguments in this section are not valid for the case where $8$ divides $ |G|$.

\section{The Classification of Regular Subgroups of $\Hol(Z_{2^{n}})$}

The purpose of this section is to prove Theorem~\ref{ch6maintheo}. Let $G=\langle a\rangle\cong Z_{2^{n}}$. We begin this section with several number theoretic results. Let $m\in \mathbb{N}$. Let $p$ be a prime.  We will  write $m=m_{\bf p}m_{{\bf p}'}$ where $m_{\bf p}$ is the $p$-part of $m$ and $(m_{\bf p}, m_{{\bf p}'})=1$. Let $G^{*}$ be the multiplicative group of the integers modulo $2^{n}$. Assuming $2^{n}>5$, then clearly $5\in G^{*}$ as $(2, 5)=1$. Let $5^{-1}$ be the inverse of $5$ in $G^{*}$.  

\begin{llemma}\cite[Theorem 2.2.6]{MR2014408}\label{ch6lem00}
Let $t\in \mathbb{N}$. Then  
$$
5^{2^{t}}\equiv 1 \Mod {2^{t+2}},
$$
and 
$$
5^{2^{t}}\not\equiv 1 \Mod {2^{t+3}}.
$$

\end{llemma}

\begin{llemma}\label{ch6lem01} %\cite{MR2488141} \cmt{which result of \cite{MR2488141}  is this? I think it needs some deduction}
Let 
\begin{center}
$M=\frac{1-5^{-kj}}{1-5^{-j}}$ and $L=\frac{1-5^{-kj}}{1+5^{-j}}$
\end{center} 
with $k\geqslant  1$ and $j\geqslant  1$. Then $M_{\bf 2}=k_{\bf 2}$ and $L_{\bf 2}=2k_{\bf 2}j_{\bf 2}$.
\end{llemma}

\begin{proof} 
Notice that $M=\frac{5^{kj}-1}{5^{kj-j}(5^{j}-1)}$ as we may multiply the numerator and the denominator of $M$ by $5^{kj}$. Thus $M_{2}=\left(\frac{5^{kj}-1}{5^{kj-j}(5^{j}-1)}\right)_{2}=\left(\frac{5^{kj}-1}{5^{j}-1}\right)_{2}$. Similarly, we have that $L_{2}=\left(\frac{5^{kj}-1}{5^{j}+1}\right)_{2}$.

Let $q=5^{j}$. Thus by \cite[Lemma 3.4]{MR2488141} we have that if $k$ is odd, then $M_{2}=1=k_{2}$. If $k$ is even, then 
$$
(q^{k}-1)_{2}=(q^{2}-1)_{2}k_{2}/2=(q-1)_{2}[(q+1)_{2}k_{2}/2]=(q-1)_{2}k_{2}
$$
as $(q+1)_{2}=2$. Thus $M_{2}=k_{2}$. Similarly, 
$$
(q^{k}-1)_{2}=(q^{2}-1)_{2}k_{2}/2=(q+1)_{2}[(q-1)_{2}k_{2}/2].
$$
If $j$ is odd, then $(q-1)_{2}=(5-1)_{2}=4$, and so $L_{2}=2k_{2}$. If $j$ is even, then 
$$
(q-1)_{2}=(5^{2}-1)_{2}j_{2}/2=4j_{2},
$$
which implies that $L_{2}=2k_{2}j_{2}$, and so  completes the proof. 
\end{proof}

\subsection{Preliminaries}
We start with presenting some basic properties of the elements in $H=\Hol(G)=G_R\rtimes \Aut(G)$. For $g\in G$, we also use $g$ to represent the corresponding element of $G_R$. Recall that  $\Aut(G)=\langle x\rangle \times \langle y\rangle$ where $x: a\to a^{-1}$ and $y: a\to a^{5}$.

\begin{llemma}\label{ch6lem11}
Let $h=a^{\alpha}x^{\beta}y^{\gamma}\in H$ where $0\leqslant  \alpha\leqslant  2^{n}-1$, $0\leqslant  \beta\leqslant  1$ and $0\leqslant  \gamma\leqslant  2^{n-2}-1$. 
\begin{enumerate}
\item If $h=a^{\alpha}y^{\beta}$, then $$h^{r}=a^{\alpha\left(\frac{1-5^{-r\gamma}}{1-5^{-\gamma}}\right)}y^{r\gamma}.$$

\item If $h=a^{\alpha}x$, then $h$ is an involution.

\item If $h=a^{\alpha}xy^{\gamma}$ with $\gamma\neq 0$, then 
\begin{equation*}
h^{r}= \begin{cases}
               a^{\alpha\sum\limits_{s=0}^{r-1}(-1)^{s}5^{-s\gamma}}xy^{r\gamma}              &  \textit{if $r$ is odd}, \\
               a^{\alpha\left(\frac{1-5^{-r\gamma}}{1+5^{-\gamma}}\right)}y^{r\gamma}               & \textit{if $r$ is even}.
               \end{cases}
\end{equation*}
\end{enumerate}
\end{llemma}

\begin{proof}  
\begin{enumerate}
\item Suppose that  $h=a^{\alpha}y^{\gamma}$. Then 
\begin{equation*}
h^{2}=a^{\alpha}y^{\gamma}a^{\alpha}y^{\gamma}=a^{\alpha}(a^{\alpha})^{y^{-\gamma}}y^{2\gamma}=a^{\alpha}a^{\alpha5^{-\gamma}}y^{2\gamma}=a^{\alpha(1+5^{-\gamma})}y^{2\gamma}
\end{equation*}
Suppose that
\begin{equation*}
h^{r}=a^{\alpha\sum\limits_{s=0}^{r-1}5^{-s\gamma}}y^{r\gamma}
\end{equation*}
holds for $r\leqslant  \ell$. Let $r=\ell+1$, and so 
\begin{align*}
h^{r+1} = & a^{\alpha\sum\limits_{s=0}^{r-1}5^{-s\gamma}}y^{r\gamma}a^{\alpha}y^{\gamma}   \\   
   = & a^{\alpha\sum\limits_{s=0}^{r-1}5^{-s\gamma}}(a^{\alpha})^{y^{-r\gamma}}y^{(r+1)\gamma}\\
   = & a^{\alpha\sum\limits_{s=0}^{r-1}5^{-s\gamma}}a^{\alpha5^{-r\gamma}}y^{(r+1)\gamma} \\
   = & a^{\alpha\sum\limits_{s=0}^{r}5^{-s\gamma}}y^{(r+1)\gamma}.
\end{align*}
Thus we have that
\begin{equation*}
h^{r}=a^{\alpha\sum\limits_{s=0}^{r-1}5^{-s\gamma}}y^{r\gamma}.
\end{equation*}
Since 
\begin{center}
$\sum\limits_{s=0}^{r-1}5^{-s\gamma}=\frac{1-5^{-r\gamma}}{1-5^{-\gamma}}$,
\end{center}
we have that 
\begin{equation}\label{ch6equ1}
h^{r}=a^{\alpha\left(\frac{1-5^{-r\gamma}}{1-5^{-\gamma}}\right)}y^{r\gamma}.
\end{equation}

\item Suppose that  $h=a^{\alpha}x$. Thus $h^{2}=a^{\alpha}xa^{\alpha}x=a^{\alpha}a^{-\alpha}xx=1$. Thus $a^{\alpha}x$ is an involution in $H$.

\item Suppose that  $h=a^{\alpha}xy^{\gamma}$. Thus 
\begin{align*}
h^{2}=a^{\alpha}xy^{\gamma}a^{\alpha}xy^{\gamma} = & a^{\alpha}x(a^{\alpha})^{y^{-\gamma}}xy^{2\gamma} \\
= & a^{\alpha}xa^{\alpha5^{-\gamma}}xy^{2\gamma} \\
= & a^{\alpha}a^{-\alpha5^{-\gamma}}y^{2\gamma}\\
= & a^{\alpha(1-5^{-\gamma})}y^{2\gamma},
\end{align*}
and 
\begin{align*}
h^{3} =a^{\alpha(1-5^{-\gamma})}y^{2\gamma}a^{\alpha}xy^{\gamma} & =a^{\alpha(1-5^{-\gamma})}(a^{\alpha})^{y^{-2\gamma}}xy^{3\gamma}\\
         & =a^{\alpha(1-5^{-\gamma})}a^{\alpha5^{-2\gamma}}xy^{3\gamma} \\
         & =a^{\alpha(1-5^{-\gamma}+5^{-2\gamma})}xy^{3\gamma}.
\end{align*}
Suppose that  when $r\leqslant  2\ell$ we have that 
\begin{equation*}
h^{r}=a^{\alpha\sum\limits_{s=0}^{r-1}(-1)^{s}5^{-s\gamma}}y^{r\gamma}
\end{equation*} 
for $r$ even, and 
\begin{equation*}
h^{r}=a^{\alpha\sum\limits_{s=0}^{r-1}(-1)^{s}5^{-s\gamma}}xy^{r\gamma}
\end{equation*} 
for $r$ odd. Now we check $r=2\ell+1$ and $r=2(\ell+1)$. If $r=2\ell+1$, then 
\begin{align*}
h^{r}=h^{r-1}a^{\alpha}xy^{\gamma} & =a^{\alpha\sum\limits_{s=0}^{r-2}(-1)^{s}5^{-s\gamma}}y^{(r-1)\gamma}a^{\alpha}xy^{\gamma}\\                          & =a^{\alpha\sum\limits_{s=0}^{r-2}(-1)^{s}5^{-s\gamma}}a^{\alpha5^{-(r-1)\gamma}}xy^{r\gamma} \\
& =a^{\alpha\sum\limits_{s=0}^{r-1}(-1)^{s}5^{-s\gamma}}xy^{r\gamma}.
\end{align*}
If $r=2(\ell+1)$, then 
\begin{align*}
h^{r}=h^{r-2}h^{2} & =a^{\alpha\sum\limits_{s=0}^{r-3}(-1)^{s}5^{-s\gamma}}y^{(r-2)\gamma}a^{\alpha(1-5^{-\gamma})}y^{2\gamma}\\                    & =a^{\alpha\sum\limits_{s=0}^{r-3}(-1)^{s}5^{-s\gamma}}a^{[\alpha(1-5^{-\gamma})]5^{-(r-2)\gamma}}y^{r\gamma}\\
& =a^{\alpha\sum\limits_{s=0}^{r-1}(-1)^{s}5^{-s\gamma}}y^{r\gamma}. 
\end{align*}
If $r$ is even, then 
\begin{center}
$\sum\limits_{s=0}^{r-1}(-1)^{s}5^{-s\gamma}=\frac{1-5^{-r\gamma}}{1+5^{-\gamma}}$
\end{center}
Hence we have that 
\begin{equation}\label{ch6equ2}
h^{r}=a^{\alpha\left(\frac{1-5^{-r\gamma}}{1+5^{-\gamma}}\right)}y^{r\gamma},
\end{equation}
when $r$ is even, and 
\begin{equation}\label{ch6equ3}
h^{r}=a^{\alpha\sum\limits_{s=0}^{r-1}(-1)^{s}5^{-s\gamma}}xy^{r\gamma}
\end{equation}
when $r$ is odd. \qedhere
\end{enumerate}
\end{proof}

\begin{llemma}\label{ch6lem12}
Let $h\in H$, and suppose that  $0\leqslant  \alpha\leqslant  2^{n}-1$, $0\leqslant  \beta\leqslant  1$ and $0\leqslant  \gamma\leqslant  2^{n-2}-1$. 
\begin{enumerate}
\item If $h=a^{\alpha}y^{\gamma}$, then $|h|=max(\frac{2^{n-2}}{\gamma_{\bf 2}}, \frac{2^{n}}{\alpha_{\bf 2}})$;
%\item If $h=a^{i}x$, then $|h|=2$;
\item If $h=a^{\alpha}xy^{\gamma}$ with $\gamma\neq 0$, then $|h|=\frac{2^{n-1}}{\gamma_{\bf 2}(\alpha, 2)}$.
\end{enumerate}
\end{llemma}

\begin{proof} {\bf (Case 1) }Suppose that  $h=a^{\alpha}y^{\gamma}$ and $h^{r}=1$ for some $r$. Since $|H|=2^{2n-1}$, we have that $r=2^{t}=r_{\bf2}$ for some $1\leqslant  t\leqslant  2n-1$. By Lemma~\ref{ch6lem11}(1) we have that $2^{n-2}\bigm| r\gamma$ and   $2^{n}\bigm| \alpha\left(\frac{1-5^{r\gamma}}{1-5^{\gamma}}\right)$. By Lemma~\ref{ch6lem01} we have that $\left(\frac{1-5^{r\gamma}}{1-5^{\gamma}}\right)_{\bf 2}=r_{\bf 2}=r$. Thus $2^{n-2}\bigm| r\gamma$ and $2^{n}\bigm| \alpha_{\bf 2}r$. Hence $max(\frac{2^{n-2}}{\gamma_{\bf 2}}, \frac{2^{n}}{\alpha_{\bf 2}})\mid r$. This implies that $|h|=max(\frac{2^{n-2}}{\gamma_{\bf 2}}, \frac{2^{n}}{\alpha_{\bf 2}})$, in particular, if $\alpha$ is odd, then $|h|=2^{n}$. 

{\bf (Case 2)} Suppose that  $h=a^{\alpha}xy^{\gamma}$ where $1\leqslant  \gamma\leqslant  2^{n-2}-1$, and $h^{r}=1$. Similarly we have  $r=2^{t}$ for some $1\leqslant  t\leqslant  2n-1$. Since $r$ is even, by Lemma~\ref{ch6lem11}(3)  we have that $h^{r}=a^{\alpha\left(\frac{1-5^{-r\gamma}}{1+5^{-\gamma}}\right)}y^{r\gamma}$, and so $2^{n}\bigm| \alpha\left(\frac{1-5^{-r\gamma}}{1+5^{-\gamma}}\right)$ and $2^{n-2}\bigm| r\gamma$. By Lemma~\ref{ch6lem01}, we have that $\left(\frac{1-5^{-r\gamma}}{1+5^{-\gamma}} \right)_{\bf 2}=2r_{\bf 2}\gamma_{\bf 2}$. Thus $2^{n}\bigm| 2r_{\bf 2}\gamma_{\bf 2}\alpha_{\bf 2}$, and so we have that 
\begin{equation}\label{ch6equlem134}
2^{n-1}\bigm| r_{\bf 2}\gamma_{\bf 2}\alpha_{\bf 2}
\end{equation}
and since $2^{n-2}\bigm| r\gamma$, we have
\begin{equation}\label{ch6equlem135}
2^{n-2}\bigm| r_{\bf 2}\gamma_{\bf 2}
\end{equation}
\begin{enumerate}
\item If $\alpha$ is odd, then (\ref{ch6equlem134}) implies that $2^{n-1}\bigm| r_{\bf 2}\gamma_{\bf 2}$, and so $|h|=\frac{2^{n-1}}{\gamma_{\bf 2}}$.

\item If $\alpha$ is even, then (\ref{ch6equlem135}) implies (\ref{ch6equlem134}). Thus we may conclude that $|h|=\frac{2^{n-2}}{\gamma_{\bf 2}}$. \qedhere
\end{enumerate} 
\end{proof}

\begin{llemma}\label{ch6lem13}
Let $h\in H$, and suppose that  $h=a^{\alpha}x^{\beta}y^{\gamma}$ with $0\leqslant  \alpha\leqslant  2^{n}-1$, $0\leqslant  \beta\leqslant  1$ and $0\leqslant  \gamma\leqslant  2^{n-2}-1$. Then $h$ is conjugate in $H$ to $a^{\alpha_{\bf 2}}x^{\beta}y^{\gamma}$. 
\end{llemma}

\begin{proof} Write $\alpha=\alpha_{\bf 2}\alpha_{{\bf2}'}$. Then there exists $\rho\in \Aut(G)$ such that $a^{\rho}=a^{\alpha_{{\bf2}'}}$ as $\alpha_{{\bf2}'}$ is odd. Since $\Aut(G)=\langle x, y\rangle$ is abelian, we have that 
\begin{equation*}
\rho h\rho^{-1}=(\rho a^{\alpha_{\bf 2}\alpha_{{\bf2}'}}\rho^{-1})x^{\beta}y^{\gamma}=(a^{\alpha_{\bf 2}\alpha_{{\bf2}'}})^{\rho^{-1}}x^{\beta}y^{\gamma}=a^{\alpha_{\bf 2}}x^{\beta}y^{\gamma},   
\end{equation*}
and the result follows. \end{proof}

Now we assemble some results that enable us to write subgroups of $\Hol(G)$ in a convenient way.

\begin{llemma}\label{ch6lem21}
Let $1\leqslant  \alpha, \alpha'\leqslant  2^{n}-1$ and $1\leqslant  \gamma, \gamma'\leqslant  2^{n-2}-1$, and $L=\langle a^{\alpha}y^{\gamma}, a^{\alpha'}xy^{\gamma'}\rangle$. Then up to conjugacy, for some $\xi$, one of the following holds:
\begin{enumerate}
\item $L=\langle  a^{\xi}x, a^{\alpha}y^{\gamma}\rangle$,
\item $L=\langle a^{\xi}, a^{\alpha'}xy^{\gamma'}\rangle$.
\end{enumerate}

\end{llemma}

\begin{proof} Since $\langle y\rangle\cong Z_{2^{n-2}}$, we have that either $\langle y^{\gamma}\rangle\geqslant \langle y^{\gamma'}\rangle$, or $\langle y^{\gamma}\rangle<\langle y^{\gamma'}\rangle$. 

Suppose that  $\langle y^{\gamma}\rangle\geqslant \langle y^{\gamma'}\rangle$. Thus there exists some $1\leqslant t\leqslant2^{n-2}-1$ such that $y^{\gamma'}=(y^{\gamma})^{t}$. Let $u=(a^{\alpha}y^{\gamma})^{t}$, and so by Lemma~\ref{ch6lem11}, 
\begin{equation*}
u=a^{M}y^{\gamma\cdot t}=a^{M}y^{\gamma'}
\end{equation*}
with $M=\alpha\cdot\frac{1-5^{-\gamma\cdot t}}{1-5^{-\gamma}}$. Thus 
\begin{align*}
u^{-1}(a^{\alpha'}xy^{\gamma'}) & = a^{M'}y^{-\gamma'}a^{\alpha'}xy^{\gamma'} \\
             & = a^{M'}(a^{\alpha'})^{y^{\gamma'}}x \\
             & = a^{M'}a^{M''}x
\end{align*}
with $M''=\alpha'\cdot 5^{\gamma'}$. Let $\xi=M'+M''$, and so $L=\langle a^{\xi}x, a^{\alpha}y^{\gamma}\rangle$. 

Now suppose that $\langle y^{\gamma}\rangle<\langle y^{\gamma'}\rangle$. Thus $\langle y^{\gamma}\rangle\leqslant\langle y^{2\gamma'}\rangle$, that is, there exists some $1\leqslant t\leqslant2^{n-2}-1$ such that $y^{\gamma}=(y^{2\gamma'})^{t}$. Since $x^{2}={\bf1}$ and $x, y$ commute, we have that 
\begin{equation*}
y^{\gamma}=(y^{2\gamma'})^{t}=((xy^{\gamma'})^{2})^{t}=(xy^{\gamma'})^{2t}.
\end{equation*}
Let $u=(a^{\alpha'}xy^{\gamma'})^{2t}$, and so by Lemma~\ref{ch6lem11}(3) we have that $u=a^{M}y^{\gamma}$ with $M=\alpha'\cdot \frac{1-5^{-\gamma'\cdot2t}}{1+5^{-\gamma'}}$. Thus
\begin{align*}
a^{\alpha}y^{\gamma}u^{-1} & = a^{\alpha}y^{\gamma}y^{-\gamma}a^{-M}\\
                & = a^{\alpha-M}.
\end{align*}
Let $\xi=\alpha-M$, and so we have that $L=\langle a^{\xi}, a^{\alpha'}xy^{\gamma'}\rangle$. \end{proof}

\begin{llemma}\label{ch6lem22}
Let $1\leqslant  \alpha, \alpha'\leqslant  2^{n}-1$ and $1\leqslant  \gamma\leqslant  2^{n-2}-1$, and $L=\langle a^{\alpha}x, a^{\alpha'}xy^{\gamma}\rangle$. Then $L=\langle a^{\alpha-\alpha'}y^{\gamma}, a^{\alpha'}xy^{\gamma}\rangle$.
\end{llemma}

\begin{proof} Since 
$$
(a^{\alpha}x)(a^{\alpha'}xy^{\gamma})=a^{\alpha-\alpha'}y^{\gamma},
$$
we have that $L=\langle a^{\alpha-\alpha'}y^{\gamma}, a^{\alpha'}xy^{\gamma}\rangle$. \end{proof}

\begin{llemma}\label{ch6lem23}
Let $L=\langle a^{\alpha}y^{\gamma}, a^{\alpha'}y^{\gamma'}\rangle$ for some $1\leqslant  \alpha, \alpha'\leqslant  2^{n}-1$ and $1\leqslant  \gamma, \gamma'\leqslant  2^{n-2}-1$. Suppose that $(\gamma)_{\bf2}\leqslant  (\gamma')_{\bf2}$. Then  $L=\langle a^{\alpha}y^{\gamma}, a^{\xi}\rangle$ for some $\xi$.
\end{llemma}

\begin{proof} Let $r=\frac{\gamma'_{\bf2}}{\gamma_{\bf2}}\cdot \gamma^{-1}_{{\bf2}'}\gamma'_{{\bf2}'}$. Then we have that 
\begin{align*}
(a^{\alpha}y^{\gamma})^{r} & =  a^{\alpha\cdot\frac{1-5^{-r\gamma}}{1-5^{-\gamma}}}y^{\gamma\cdot r} \\
   & = a^{\alpha\cdot\frac{1-5^{-r\gamma}}{1-5^{-\gamma}}}y^{\gamma'} \\
   & = a^{M}y^{\gamma'},
\end{align*}
where $M=\alpha\cdot\frac{1-5^{-r\gamma}}{1-5^{-\gamma}}$. Thus 
\begin{align*}
(a^{\alpha}y^{\gamma})^{r}\cdot (a^{\alpha'}y^{\gamma'})^{-1} & = a^{M}y^{\gamma'}\cdot y^{-\gamma'}a^{-\alpha'} \\ 
   & = a^{M-\alpha'} \\
   & = a^{\xi}
\end{align*}
with $\xi=M-\alpha'$. Thus we have that $L=\langle a^{\alpha}y^{\gamma}, a^{\xi}\rangle$.  \end{proof}

\begin{llemma}\label{ch6lem24}
Let $1\leqslant  \alpha, \alpha'\leqslant  2^{n}-1$ and $1\leqslant  \gamma, \gamma'\leqslant  2^{n-2}-1$, and $L=\langle a^{\alpha}xy^{\gamma}, a^{\alpha'}xy^{\gamma'}\rangle$. Then up to conjugacy, for some $M$ and $\xi$, one of the following holds: 
\begin{enumerate}
\item $L=\langle a^{\xi}x, a^{M}y^{\gamma+\gamma'}\rangle$, 
\item $L=\langle a^{\xi}, a^{\alpha'}xy^{\gamma'}\rangle$.
\end{enumerate}  
\end{llemma}

\begin{proof} Since 
\begin{align*}
(a^{\alpha}xy^{\gamma})\cdot (a^{\alpha'}xy^{\gamma'}) & = a^{\alpha}y^{\gamma}a^{-\alpha'}y^{\gamma'}
\\
   & = a^{\alpha-\alpha'\cdot 5^{-\gamma}}y^{\gamma+\gamma'},
\end{align*}
we have that $L=\langle a^{M}y^{\gamma+\gamma'}, a^{\alpha'}xy^{\gamma'}\rangle$ with $M=\alpha-\alpha'\cdot 5^{-\gamma}$. By Lemma~\ref{ch6lem21} we have that either $L=\langle a^{\xi}x, a^{M}y^{\gamma+\gamma'}\rangle$, or $L=\langle a^{\xi}, a^{\alpha'}xy^{\gamma'}\rangle$ for some $\xi$. \end{proof}

\subsection{Semiregular Elements}

\begin{defin}
A group $X$ is \emph{semiregular} on a finite set $\Omega$ if $X_{\omega}=\bf1$ for all $\omega\in \Omega$. We say that $g\in X$ is a \emph{semiregular element} (or \emph{semiregular}) if $\langle g\rangle$ is semiregular on $\Omega$.
\end{defin}

Recall that $G=\langle a\rangle\cong Z_{2^{n}}$ and $H=\Hol(G)$. Let $R$ be a subgroup of $H$. If $R$ is regular, then each element of $R$ is semiregular. Thus it is necessary to study the semiregular elements of $H$ and we classify all such elements in this section.

Let $g=a^{\epsilon}\in G$ where $1\leqslant  \epsilon\leqslant  2^{n}-1$. We now view the elements of $G$ as the elements of $G_R\leqslant\Hol(G)$ and of the set $G$ that $\Hol(G)$ acts on. Precisely, this means that for $h=a^{\alpha}x^{\beta}y^{\gamma}\in\Hol(G)$ we have that 
$$
g^{h}=(a^{\epsilon})^{a^\alpha x^{\beta}y^{\gamma}}=(a^{\epsilon+\alpha})^{x^{\beta}y^{\gamma}}.
$$
Moreover, $H_g$ denotes the stabiliser of $g\in\Omega$ in $H$.

\begin{llemma}\label{ch6lem300}
Let $g\in G$ and $g\neq {\bf1}$. Then $H_{g}=\langle g^{-2}x, g^{5^{-1}-1}y\rangle$.
\end{llemma}

\begin{proof} Since $H_{\bf1}=\langle x, y\rangle$, we have that $H_{g}  = H^{g}_{\bf1}=\langle g^{-2}x, g^{5^{-1}-1}y\rangle$. \end{proof}

\begin{llemma}\label{ch6lem301}
Let $h\in H$ and $h\neq {\bf1}$. Then $h$ is semiregular if and only if $h^{\frac{|h|}{2}}$ is semiregular. 
\end{llemma}

\begin{proof} Recall that if $h$ is not semiregular, then there exists some $g\in G$ and some $k$ such that $g^{h^{k}}=g$ and $h^{k}\neq{\bf1}$. Moreover, any power of $h^{k}$ fixes $g$. Clearly if $h$ is semiregular, then $h^{\frac{|h|}{2}}$ is semiregular. Now suppose that $h^{\frac{|h|}{2}}$ is semiregular. Since $\langle h^{\frac{|h|}{2}}\rangle\leqslant \langle h^{k}\rangle$ for all $1\leqslant k\leqslant |h|-1$, we have that $h$ is semiregular. \end{proof}

Let $h\in H$, and suppose that  $h=a^{\alpha}x^{\beta}y^{\gamma}$ where $0\leqslant  \alpha\leqslant  2^{n}-1$, $0\leqslant  \beta\leqslant  1$ and $0\leqslant  \gamma\leqslant  2^{n-2}-1$. By Lemma~\ref{ch6lem13} we have that $h$ is conjugate to $a^{\alpha_{\bf2}}x^{\beta}y^{\gamma}$ where $\alpha_{\bf2}=2^{t}$ for some $0\leqslant  t\leqslant  n$. Note that if two elements are conjugate, then either they are both semiregular, or both non-semiregular. Thus it is sufficient to check which $a^{2^{t}}x^{\beta}y^{\gamma}$ are semiregular with $0\leqslant  t\leqslant  n$. Let $h=a^{2^{t}}x^{\beta}y^{\gamma}$. We may observe that if $t=n$, that is, $h=x^{\beta}y^{\gamma}$, then $h$ is semiregular if and only if $h=\bf1$. Thus we may assume that $0\leqslant  t\leqslant  n-1$. Also we know that if $h$ is semiregular, then $h^{r}$ fixes a group element in $G$ if and only if $h^{r}=\bf1$.

\begin{llemma}\label{ch6lem31}
Suppose that  $h=a^{2^{t}}y^{\gamma}$ where $0\leqslant  t\leqslant  n-1$ and $1\leqslant  \gamma\leqslant  2^{n-2}-1$. Then $h$ is semiregular if and only if $h=a^{2^{t}}y^{\gamma}$ with $2^{t}<4\gamma_{\bf2}$.
\end{llemma}

\begin{proof}  First we show that $h$ is semiregular if $2^{t}<4\gamma_{\bf2}$. Since $2^{t}< 4\gamma_{\bf2}$, we have that $\gamma_{\bf2}\geqslant2^{t-1}$ and so by Lemma~\ref{ch6lem12}, $|h|=2^{n-t}$. Thus 
\begin{align*}
h^{\frac{|h|}{2}} = h^{2^{n-t-1}} & = (a^{2^{t}}y^{\gamma})^{2^{n-t-1}}\\
                           & = a^{M}y^{\gamma\cdot 2^{n-t-1}}\\
                           & = a^{M}
\end{align*}
with $M=2^{t}\frac{1-5^{-2^{n-t-1}\gamma}}{1-5^{-\gamma}}$. By Lemma~\ref{ch6lem01} we have that $M_{\bf2}=2^{n-1}$ and so $h^{\frac{|h|}{2}}=a^{2^{n-1}}\in G_R$. Thus $h^{\frac{|h|}{2}}$ is semiregular, and so by Lemma~\ref{ch6lem301} we have that  $h$ is semiregular when $2^{t}<4\gamma_{\bf2}$. 

Now suppose that $2^{t}\geqslant 4\gamma_{\bf2}$, and so by Lemma~\ref{ch6lem12}, $|h|=\frac{2^{n-2}}{\gamma_{\bf2}}$. Thus 
\begin{align*}
h^{\frac{|h|}{2}} & = (a^{2^{t}}y^{\gamma})^{\frac{2^{n-3}}{\gamma_{\bf2}}} \\
                           & = a^{M}y^{2^{n-3}\gamma_{{\bf2}'}}
\end{align*}
with $M=2^{t}\frac{1-5^{-\gamma(\frac{2^{n-3}}{\gamma_{\bf2}})}}{1-5^{-\gamma}}$, and so by Lemma~\ref{ch6lem01},  $M_{\bf2}=\frac{2^{n+t-3}}{\gamma_{\bf2}}$. Observe that $y^{2^{n-3}\gamma_{{\bf2}'}}=y^{2^{n-3}}$ as $\gamma_{{\bf2}'}$ is odd and $|y^{2^{n-3}}|=2$. Thus $h^{\frac{|h|}{2}}=a^{M}y^{2^{n-3}}$. If $2^{t}>4\gamma_{\bf2}$, then 
\begin{equation*}
M_{\bf2}> 2^{n-1},
\end{equation*}
which implies that $2^{n}\mid M$. Thus $h^{\frac{|h|}{2}}=y^{2^{n-3}}\in H_{\bf1}$, and by Lemma~\ref{ch6lem301} we have that $h$ is not semiregular. 

If $2^{t}=4\gamma_{\bf2}$, then $h^{\frac{|h|}{2}}=a^{2^{n-1}}y^{2^{n-3}}$. Since 
\begin{align*}
a^{a^{2^{n-1}}y^{2^{n-3}}} & = a^{(1+2^{n-1})\cdot 5^{2^{n-3}}} \\
                                          & = a^{(5^{2^{n-3}}-1)+2^{n-1}\cdot 5^{2^{n-3}}+1}
\end{align*}
and by Lemma~\ref{ch6lem00}, we have that $5^{2^{n-3}}\equiv 1 \Mod {2^{n-1}}$, and so 
\begin{equation}\label{ch6equlem31}
a^{a^{2^{n-1}}y^{2^{n-3}}}=a.
\end{equation} 
Thus $h^{\frac{|h|}{2}}$ is not semiregular when $2^{t}=4\gamma_{\bf2}$. Hence  $h$ is not semiregular when $2^{t}\geqslant 4\gamma_{\bf2}$, and so we conclude that $h$ is semiregular if and only if $2^{t}< 4\gamma_{\bf2}$. \end{proof}

\begin{llemma}\label{ch6lem32}
Suppose that  $h=a^{2^{t}}x$ with $0\leqslant  t\leqslant  n-1$. Then $h$ is semiregular if and only if $h=ax$.
\end{llemma}

\begin{proof}  By Lemma~\ref{ch6lem300}, $H_{a^{-2^{t-1}}}=\langle a^{2^{t}}x, a^{-2^{t-1}(5^{-1}-1)}\rangle$,  and so  if $t\geqslant 1$, then $h$ fixes $a^{-2^{t-1}}$, and so $h$ is not semiregular. If $t=0$, then since $h$ is an involution and interchanges the cosets $\langle a^{2}\rangle$ and $\langle a^{2}\rangle a$, $h$ is semiregular.  \end{proof}

\begin{llemma}\label{ch6lem33}
Suppose that  $h=a^{2^{t}}xy^{\gamma}$ with $0\leqslant t\leqslant n$ and $1\leqslant \gamma\leqslant 2^{n-2}-1$. Then $h$ is semiregular if and only if $h=axy^{\gamma}$. 
\end{llemma}

\begin{proof}  First suppose that $t=0$, and so $|h|=\frac{2^{n-1}}{\gamma_{\bf2}}$ by Lemma~\ref{ch6lem12}. Since $\frac{|h|}{2}$ is even, by Lemma~\ref{ch6lem11}(3) we have that  
\begin{equation*}
h^{\frac{|h|}{2}}= (axy^{\gamma})^{\frac{2^{n-2}}{\gamma_{\bf2}}}=a^{M}
\end{equation*}
with $M=\frac{1-5^{-\gamma\cdot\frac{2^{n-2}}{\gamma_{\bf2}}}}{1+5^{-\gamma}}$ by Lemma~\ref{ch6lem11}. Thus by Lemma~\ref{ch6lem01}, $M_{\bf2}=2\gamma_{\bf2}\frac{2^{n-2}}{\gamma_{\bf2}}=2^{n-1}$, which implies that $h^{\frac{|h|}{2}}=a^{2^{n-1}}\in G$. Hence by Lemma~\ref{ch6lem301}, $h$ is semiregular when $t=0$. 

Suppose now that $t\geqslant 1$, and so $|h|=\frac{2^{n-2}}{\gamma_{\bf2}}$. Notice that $\frac{|h|}{2}$ is even if and only if $\gamma_{\bf2}<2^{n-3}$. Assume that $\gamma_{\bf2}<2^{n-3}$, then 
\begin{align*}
h^{\frac{|h|}{2}} & = a^{M}y^{\gamma\cdot\frac{2^{n-3}}{\gamma_{\bf2}}} \\
                           & = a^{M}y^{2^{n-3}}
\end{align*}                           
with $M=2^{t}\frac{1-5^{-\gamma\cdot\frac{|h|}{2}}}{1+5^{-\gamma}}$ by Lemma~\ref{ch6lem11}, and so by Lemma~\ref{ch6lem01}
$$
h^{\frac{|h|}{2}}=a^{2^{n+t-2}}y^{2^{n-3}}.
$$  
If $t\geqslant 2$, then $h^{\frac{|h|}{2}}=y^{2^{n-3}}\neq {\bf1}$ and $y^{2^{n-3}}\in H_{\bf1}$, which implies that $h^{\frac{|h|}{2}}$ is not semiregular. Thus we may assume that $t=1$. Then $h^{\frac{|h|}{2}}=a^{2^{n-1}}y^{2^{n-3}}$, and so by (\ref{ch6equlem31}) we have that $h^{\frac{|h|}{2}}$ fixes $a$. Thus $h^{\frac{|h|}{2}}$ is not semiregular, and by Lemma~\ref{ch6lem301} we have that $h$ is not semiregular. Hence $h$ is not semiregular when $\gamma_{\bf2}<2^{n-3}$ and $t\geqslant 1$.

Now it remains to consider the case where $\gamma_{\bf2}=2^{n-3}$ and $t\geqslant 1$. Here $|h|=2$. By Lemma~\ref{ch6lem300} we have that
\begin{equation*}
H_{a^{2^{t-1}}}=\langle a^{-2^{t}}x, a^{2^{t-1}(5^{-1}-1)}y\rangle.
\end{equation*} 
Thus $(a^{2^{t-1}(5^{-1}-1)}y)^{2^{n-3}}\in H_{a^{2^{t-1}}}$, and so $H_{a^{2^{t-1}}}$ also contains
\begin{align*}
a^{-2^{t}}x[(a^{2^{t-1}(5^{-1}-1)}y)^{2^{n-3}}] & = a^{-2^{t}}x[a^{2^{t-1}(5^{-1}-1)\frac{1-5^{-2^{n-3}}}{1-5^{-1}}}y^{2^{n-3}}] \\
       & = a^{-2^{t}}x[a^{2^{t-1}(5^{-2^{n-3}}-1)}y^{2^{n-3}}]\\
       & = a^{-2^{t}}xa^{2^{t-1}2\ell}y^{2^{n-3}} \\
       & = a^{-2^{t}}a^{-2^{t}\ell}xy^{2^{n-3}}  \\
       & = a^{2^{t}M}xy^{2^{n-3}}
\end{align*}
with $M=-(1+\ell)$ and $5^{-2^{n-3}}-1=2\ell$. Note that by Lemma~\ref{ch6lem00}, $5^{-2^{n-3}}\equiv 1 \Mod {2^{n-1}}$, and so $\ell$ is even and $M$ is odd. Since $a^{2^{t}M}xy^{2^{n-3}}\neq \bf1$, we have that $a^{2^{t}M}xy^{2^{n-3}}$ is not semiregular. Moreover, by Lemma~\ref{ch6lem13}, $a^{2^{t}M}xy^{2^{n-3}}$ and $h=a^{2^{t}}xy^{2^{n-3}}$ are conjugate. Thus $h$ is not semiregular when $t\geqslant 1$ and $\gamma_{\bf2}=2^{n-3}$, and so $h$ is not semiregular with $t\geqslant 1$. Hence $h$ is semiregular if and only if $h=axy^{\gamma}$. \end{proof}

Notice that $G_R$ acts regularly on $\Omega=G$, and so $g$ is semiregular for all $g\in G_R$. Since all elements of $G_R$ of the same order are conjugate in $\Hol(G)$, together with Lemma~\ref{ch6lem31}, Lemma~\ref{ch6lem32} and Lemma~\ref{ch6lem33} we obtain the following theorem.

\begin{guess}\label{ch6them31}
Let $h\in H$. If $n\geqslant 3$, then $h$ is semiregular if and only if $h$ is conjugate to one of
\begin{enumerate}
\item $a^{2^{t}}$ with $0\leqslant t\leqslant n-1$,
\item $a^{2^{t}}y^{\gamma}$ with $2^{t}<4\gamma_{\bf2}$ and  $1\leqslant \gamma\leqslant 2^{n-2}-1$,
\item $axy^{\gamma}$ with $0\leqslant \gamma\leqslant2^{n-2}-1$.
\end{enumerate}
\end{guess}

\subsection{Proof of Theorem~\ref{ch6maintheo}}

%Let $R$ be a regular subgroup of $H$. By the Second Isomorphism Theorem, we have that 
%$$
%R/(R\cap G_{R})\cong RG_{R}/G_{R}\leqslant H/G_{R}=\Aut(G).
%$$
%Since $|R|=2^{n}$ and $|\Aut(G)|=2^{n-1}$, we have that $R\cap G_{R}\neq \bf1$. We may assume that $R\cap G_{R}=\langle a^{\alpha}\rangle$ for some even $\alpha$. 

The next result provides a necessary condition for a subgroup of $H$ to be regular.

\begin{guess}\label{ch6them41}
If $R$ is a regular subgroup of $H$, then up to conjugacy, $R$ has exactly one of the following types:
\begin{enumerate}
\item $R=G_{R}$; 
\item $R=(R\cap G_{R})\langle  a^{2^{t}}y^{\gamma}\rangle$ where $2^{t}< 4\gamma_{\bf2}$ and $1\leqslant \gamma\leqslant 2^{n-2}-1$;
\item $R=(R\cap G_{R})\langle  ax\rangle$;
\item $R=(R\cap G_{R})\langle  axy^{\gamma}\rangle$ where $1\leqslant \gamma\leqslant 2^{n-2}-1$;
\item $R=(R\cap G_{R})\langle  ax, a^{\epsilon}y^{\gamma}\rangle$ where $2\leqslant \epsilon_{\bf2}< 4\gamma_{\bf2}$ and $1\leqslant \gamma\leqslant 2^{n-2}-1$.
\end{enumerate}
\end{guess}

\begin{proof} By the Second Isomorphism Theorem, we have that  
\begin{equation*}
\faktor{R}{R\cap G_{R}}\cong \faktor{RG_{R}}{G_{R}}\leqslant \faktor{H}{G_{R}}=\Aut(G)=\langle x\rangle \times\langle y\rangle.
\end{equation*}
Thus either $R=G_{R}$, or $R=\langle R\cap G_{R}, h\rangle$ or $R=\langle R\cap G_{R}, h_{1}, h_{2}\rangle$ where $h, h_{1}, h_{2}\notin G_R$. 

Suppose first that $R=\langle R\cap G_{R}, h\rangle$. Since $R$ is semiregular, so is $h$ and so Theorem~\ref{ch6them31} implies that up to conjugacy 
\begin{enumerate}
\item $h=a^{2^{t}}y^{\gamma}$ with $2^{t}<4\gamma_{\bf2}$ and $1\leqslant \gamma\leqslant 2^{n-2}-1$;
\item $h=axy^{\gamma}$ with $0\leqslant \gamma\leqslant 2^{n-2}-1$.
\end{enumerate}
If $h=a^{2^{t}}y^{\gamma}$ where $2^{t}<4\gamma_{\bf2}$ with $1\leqslant \gamma\leqslant 2^{n-2}-1$, then we get case (2) of this theorem. Case (3) and case (4) of this theorem follow if $h=axy^{\gamma}$ with $0\leqslant \gamma\leqslant 2^{n-2}-1$. Thus it is left to consider the case where $R=\langle R\cap G_{R}, h_{1}, h_{2}\rangle$ where $h, h_{1}, h_{2}\notin G_R$ and in particular $R/(R\cap G_{R})$ is not cyclic.

Suppose that $h_{1}=a^{\alpha_{1}}x$ and $h_{2}=a^{\alpha_{2}}x^{\beta_{2}}y^{\gamma_{2}}$. Since $R$ is regular, Theorem~\ref{ch6them31} implies that $\alpha_{1}$ is odd. If $\beta_{2}=0$, then using the fact that $R/(R\cap G_{R})$ is not cyclic, $R$ is conjugate to the case (5) of this theorem. If $\beta_{2}=1$, then by Theorem~\ref{ch6them31} we have that $\alpha_{2}$ is also odd. Thus  Lemma~\ref{ch6lem22} implies that $R=\langle R\cap G_{R}, a^{\xi}y^{\gamma_{2}}, a^{\alpha_{2}}xy^{\gamma_{2}}\rangle$, and Lemma~\ref{ch6lem31} implies that $2\leqslant \xi\leqslant 4(\gamma_{2})_{\bf2}$ as $R$ is regular.  Further by Lemma~\ref{ch6lem21}, $R$ either conjugates to $\langle R\cap G_{R}, a^{\xi'}x, a^{\xi}y^{\gamma_{2}}\rangle$ which implies that $R$ is conjugate to case (5), or $R$ conjugates to $\langle R\cap G_{R}, a^{\alpha_{2}}xy^{\gamma_{2}}\rangle$, which contradicts  $R/(R\cap G_{R})$ being cyclic. 

Finally suppose that $h_{1}=a^{\alpha_{1}}x^{\beta_{1}}y^{\gamma_{1}}$ and $h_{2}=a^{\alpha_{2}}x^{\beta_{2}}y^{\gamma_{2}}$ with $\gamma_{1}\gamma_{2}\neq 0$. If $\beta_{1}+\beta_{2}=1$, say $\beta_{1}=0$ and $\beta_{2}=1$, then since $R/(R\cap G_{R})$ is not cyclic, by Lemma~\ref{ch6lem21} we have that up to conjugacy $R=\langle R\cap G_{R}, a^{\xi}x, a^{\alpha_{1}}y^{\gamma}\rangle$ for some $\xi$. Again $\xi$ is odd as $R$ is regular. Thus $R$ is conjugate to case (5) of the theorem. Now suppose that $\beta_{1}+\beta_{2}=0$. Since $R/(R\cap G_{R})$ is not cyclic, we have that $\beta_{1}=\beta_{2}=1$. Then by Lemma~\ref{ch6lem24} and using the fact that $R/(R\cap G_{R})$ is cyclic, we have that up to conjugacy $R=\langle R\cap G_{R}, a^{\xi}x, a^{M}y^{\gamma}\rangle$ for some $\xi, M$ and $\gamma$. Thus $R$ is conjugate to (5) of the theorem. \end{proof}

Observe that if $R=G_{R}$, then $R$ is regular on $G$. Thus to complete the classification of the regular subgroups of $H$ it remains to check which subgroups of types $(2)-(5)$ in Theorem \ref{ch6them41} are regular. Since in these cases $R\cap G_{R}<G_{R}$, we assume from now on that $R\cap G_{R}=\langle a^{\alpha}\rangle$ with $\alpha$ even.

\bigskip 

{\bf Type (2).}

Let $R\leqslant H$ be of type $(2)$, that is, $R=(R\cap G_{R})\langle a^{2^{t}}y^{\gamma}\rangle$ where $2^{t}< 4\gamma_{\bf2}$ and $1\leqslant \gamma\leqslant 2^{n-2}-1$. %We may assume that $R\cap G=\langle a^{i}\rangle$ for some $1\leqslant i\leqslant 2^{n}-1$ with $i$ even. 

\begin{llemma}\label{ch6lem411}
If $R$ is regular, then $R=\langle ay^{\gamma}\rangle$.
\end{llemma}

\begin{proof} Assume that $t\geqslant 1$. Thus since $\alpha$ is even, we have that $a^{a^{\alpha}}=a^{1+\alpha}$ where $1+\alpha$ is odd. Similarly $a^{a^{2^{t}}y^{\gamma}}=a^{(1+2^{t})\cdot5^{\gamma}}$ where $(1+2^{t})\cdot5^{\gamma}$ is odd. Thus $a^{R}\subseteq \{a^{\epsilon}|1\leqslant \epsilon\leqslant 2^{n-1} \ \textrm{where $\epsilon$ is odd}\}$, and so $R$ is not transitive on $G$, which leads to a contradiction. Hence $t=0$, that is, $R=(R\cap G_{R})\langle ay^{\gamma}\rangle$. Therefore $R=\langle ay^{\gamma}\rangle$ as Lemma~\ref{ch6lem12} implies that $|ay^{\gamma}|=2^{n}$.  \end{proof}

\begin{guess}\label{ch6them411}
$R$ is regular if and only if $R$ is conjugate to $\langle ay^{2^{t}}\rangle$ for some $0\leqslant t\leqslant  n-3$ ($n\geqslant 3$).  Further $R\cap G_{R}=\langle a^{2^{n-2-t}}\rangle$.
\end{guess}

\begin{proof} By Lemma~\ref{ch6lem411}, it remains to show that $\langle ay^{\gamma}\rangle$ is conjugate to $\langle ay^{2^{t}}\rangle$ for some $0\leqslant t\leqslant n-3$. 

Write $\gamma=\gamma_{\bf 2}\gamma_{{\bf2}'}$ and consider $(ay^{\gamma_{\bf 2}})^{\gamma_{{\bf2}'}}$. Then by Lemma~\ref{ch6lem11}
\begin{align*}
(ay^{\gamma_{\bf 2}})^{\gamma_{{\bf2}'}} & = a^{\frac{1-5^{-\gamma_{\bf 2}\gamma_{{\bf2}'}}}{1-5^{-\gamma_{\bf 2}}}}y^{\gamma_{\bf 2}\gamma_{{\bf2}'}}\\
   & =  a^{\frac{1-5^{-\gamma_{\bf 2}\gamma_{{\bf2}'}}}{1-5^{-\gamma_{\bf 2}}}}y^{\gamma}.
\end{align*}

By Lemma~\ref{ch6lem01}, since $\gamma_{{\bf2}'}$ is odd, we have that $(\frac{1-5^{-\gamma_{\bf 2}\gamma_{{\bf2}'}}}{1-5^{-\gamma_{\bf 2}}})_{\bf 2}=1$. Thus by Lemma~\ref{ch6lem13}, $(ay^{\gamma_{\bf 2}})^{\gamma_{{\bf2}'}} $ is conjugate to $ay^{\gamma}$, and so $\langle (ay^{\gamma_{\bf 2}})^{\gamma_{{\bf2}'}} \rangle$ is conjugate to $\langle ay^{\gamma}\rangle$. By Lemma~\ref{ch6lem12}, $\langle ay^{\gamma_{\bf2}}\rangle\cong Z_{2^{n}}$ and since $\gamma_{{\bf2}'}$ is odd, we have that $\langle ay^{\gamma_{\bf2}}\rangle=\langle (ay^{\gamma_{\bf2}})^{\gamma_{{\bf2}'}}\rangle$. Hence $\langle ay^{\gamma}\rangle$ is conjugate to $\langle ay^{\gamma_{\bf 2}}\rangle$. Note that Lemma~\ref{ch6lem11} implies that $(ay^{2^{t}})^{2^{n-2-t}}=a^{M}$, with $M=\frac{1-5^{-2^{n-2}}}{1-5^{-2^{t}}}$. Thus $\langle a^{M}\rangle\leqslant R\cap G_{R}$. Notice that for each $u\in R\cap G_{R}$, we have that $u=(ay^{2^{t}})^{\ell}$ such that $2^{n-2-t}\mid \ell$, and so $u\in \langle a^{M}\rangle$. Hence $R\cap G_{R}=\langle a^{M}\rangle$. By Lemma~\ref{ch6lem01} we have that $M_{\bf2}=2^{n-2-t}$, and so $R\cap G_{R}=\langle a^{2^{n-2-t}}\rangle$. \end{proof} 

%Further $R$ is cyclic. \end{proof}

\bigskip

{\bf Type (3).}

\begin{guess}\label{ch6them421}
Suppose that $R=(R\cap G_{R})\langle ax\rangle$. Then $R$ is regular if and only if $R=\langle a^{2}, ax\rangle$. In this case, $R\cong D_{2^{n}}$.
\end{guess}

\begin{proof} {\it (Sufficiency.)} Since $|ax|=2$, we have that $\langle ax\rangle\cap G_{R}={\bf1}$. Thus 
$$
|\langle ax, a^{\epsilon}\rangle|=\frac{2\cdot 2^{n}}{\epsilon_{\bf2}}
$$
and so if $R$ is regular we have that $\epsilon_{\bf2}=2$. Notice that $\langle a^{2\epsilon_{{\bf2}'}}\rangle =\langle a^{2}\rangle$, and so we may let $\epsilon=2$. Moreover, since $|a^{2}|=2^{n-1}$, and 
\begin{equation*}
ax(a^{2})ax=axa^{3}x=aa^{-3}=a^{-2},
\end{equation*}
we have that $R=\langle a^{2}, ax\rangle\cong D_{2^{n}}$. 

{\it (Necessity.)} It is not hard to show that $\langle a^{2}\rangle$ has two orbits on $G$, which are $O_{1}=\{a^{\epsilon}| \epsilon \text{ is odd} \}$ and $O_{2}=\{a^{\epsilon}| \epsilon \text{ is even}\}$. Observe that $a^{\langle a^{2}\rangle}=O_{1}$.  Since $a^{ax}=a^{-2}\in O_{2}$, and $\langle a^{2}\rangle$ is transitive on $O_{2}$, we may conclude that $\langle a^{2}, ax\rangle$ is transitive on $G$, which by comparing orders implies that $\langle a^{2}, ax\rangle$ is regular. Therefore, $R$ is regular if and only if $R=\langle a^{2}, ax\rangle$.   \end{proof}

\bigskip

{\bf Type (4).}

\begin{guess}\label{ch6them431}
Suppose that $R=(R\cap G_{R})\langle axy^{\gamma}\rangle$ with $1\leqslant  \gamma\leqslant  2^{n-2}-1$ ($n\geqslant 3$). Then $R$ is regular  if and only if
$R=\langle a^{2}, axy^{2^{n-3}}\rangle$. In this case, $R$ is isomorphic to $Q_{2^{n}}$.
\end{guess}

\begin{proof} {\it (Sufficiency.)}  Recall that $R\cap G_{R}=\langle a^{\alpha}\rangle$ with $\alpha$ even. 

Suppose first that $n=3$. Then $G=Z_{8}$ and $R=\langle a^{\alpha}, axy\rangle$. Since $\alpha$ is even, $\langle a^{\alpha}, axy\rangle\leqslant  \langle a^{2}, axy\rangle$ and equality holds if and only if $\alpha_{\bf2}=2$. Since $|a^{2}|=4$, $(axy)^{2}=a^{4}$ and $(axy)^{-1}a^{2}axy=a^{-2}$, we have that $\langle a^{2}, axy\rangle\cong Q_{8}$.  Since $(axy)^{2}=a^{4}$, any element of $\langle a^{2}, axy\rangle\backslash \langle a^{2}\rangle$ is of the form $a^{\epsilon}xy$ with $\epsilon$ odd, and so is conjugate to $axy$ in $\Hol(G)$. Thus $\langle a^{2}, axy\rangle$  is regular as $|\langle a^{2}, axy\rangle|=2^{3}$. Hence when $n=3$, $R$ is regular if and only if $R=\langle a^{2}, axy\rangle$.

Now suppose that $n>3$. By Lemma~\ref{ch6lem11},  $(axy^{\gamma})^{\frac{2^{n-3}}{\gamma_{\bf2}}}\notin G_R$ and $(axy^{\gamma})^{\frac{2^{n-2}}{\gamma_{\bf2}}}\in G_R$, and 
\begin{align*}
(axy^{\gamma})^{\frac{2^{n-2}}{\gamma_{\bf2}}} & = a^{\frac{1-5^{-\gamma\cdot\frac{2^{n-2}}{\gamma_{\bf2}}}}{1+5^{-\gamma}}}y^{\gamma\cdot\frac{2^{n-2}}{\gamma_{\bf2}}} \\
    & = a^{2^{n-1}},
\end{align*}
we have that $\langle axy^{\gamma}\rangle\cap G_R=\langle a^{2^{n-1}}\rangle$. Since $R$ is regular and $R=\langle a^{\alpha}, axy^{\gamma}\rangle=\langle a^{\alpha}\rangle\langle axy^{\gamma}\rangle$, we have that 
\begin{align*}
2^{n}=|\langle a^{\alpha}\rangle\langle axy^{\gamma}\rangle| & = \frac{(2^{n}/\alpha_{\bf2})(2^{n-1}/\gamma_{\bf2})}{2} \\
  & = \frac{2^{2n-2}}{\alpha_{\bf2}\gamma_{\bf2}}.
\end{align*} 
Hence $\alpha_{\bf2}\gamma_{\bf2}=2^{n-2}$. Moreover, by Lemma~\ref{ch6lem12}, $|axy^{\gamma}|=\frac{2^{n-1}}{\gamma_{\bf2}}$. 

Suppose that $\gamma_{\bf2}\leqslant 2^{n-4}$. Let $r=\frac{|axy^{\gamma}|}{4}$, and so $r$ is even and $(axy^{\gamma})^{r}=a^{M}y^{2^{n-3}}$ where $M=\frac{1-5^{-r\gamma}}{1+5^{-\gamma}}$. Since $M_{\bf2}=2\gamma_{\bf2}r_{\bf2}=2\cdot \frac{2^{n-1}}{4\gamma_{\bf2}}\cdot \gamma_{\bf2}=2^{n-2}$, we have that $(axy^{\gamma})^{r}=a^{2^{n-2}M_{{\bf2}'}}y^{2^{n-3}}$. Notice that $a^{2^{n-2}M_{{\bf2}'}}\in \langle a^{\alpha}\rangle=R\cap G_{R}$ as $2\leqslant \alpha_{\bf2}\leqslant 2^{n-2}$. This implies that $y^{2^{n-3}}\in R$, and so $R$ is not regular. Thus $\gamma_{\bf2}=2^{n-3}$ and $\alpha_{\bf2}=2$. Hence $R=\langle a^{2}, axy^{2^{n-3}}\rangle$ as required.

Since $(axy^{2^{n-3}})^{2}=a^{2^{n-1}}=(a^{2})^{2^{n-2}}$, and 
\begin{align*}
(axy^{2^{n-3}})^{-1}a^{2}(axy^{2^{n-3}}) & = (xy^{-2^{n-3}}a^{-1})a^{2}(axy^{2^{n-3}}) \\
     & =a^{5^{2^{n-3}}}xy^{-2^{n-3}}a^{2}axy^{2^{n-3}}\\
     & = a^{5^{2^{n-3}}}xa^{3\cdot5^{2^{n-3}}}x\\
     & = a^{-2\cdot5^{2^{n-3}}}.
\end{align*}
By Lemma~\ref{ch6lem00}, we have that $5^{2^{n-3}}\equiv 1 \Mod {2^{n-1}}$, and so $5^{2^{n-3}}=2^{n-1}t+1$ and $-2\cdot5^{2^{n-3}}=-2(2^{n-1}t+1)\equiv -2 \Mod {2^{n}}$. Thus $(axy^{2^{n-3}})^{-1}a^{2}(axy^{2^{n-3}})=a^{-2}$. Since $(a^{2})^{2^{n-1}}=1$ and $(axy^{2^{n-3}})^{4}=1$, we have that $R=\langle a^{2}, axy^{2^{n-3}}\rangle\cong Q_{2^{n}}$.

{\it (Necessity.)} As seen in Type (3), we have that $\langle a^{2}\rangle$ has two orbits $O_{1}$ and $O_{2}$ on $G$ and $a^{\langle a^{2}\rangle}=O_{1}$. Since $a^{axy^{2^{n-3}}}=a^{-2^{n-2}}\in O_{2}$ and $\langle a^{2}\rangle$ is transitive on $O_{2}$, we have that $\langle a^{2}, axy^{2^{n-3}}\rangle$ is transitive on $G$, which by comparing orders implies that $\langle a^{2}, axy^{2^{n-3}}\rangle$ is regular.  \end{proof}

\bigskip

{\bf Type (5).}

Suppose that $R=(R\cap G_{R})\langle a^{\epsilon}y^{\gamma}, ax\rangle$ with $2\leqslant  \epsilon_{\bf2}\leqslant  4\gamma_{\bf2}$ and  $1\leqslant  \gamma\leqslant  2^{n-2}-1$ ($n\geqslant 3$) and $R\cap G_{R}=\langle a^{\alpha}\rangle=\langle a^{\alpha_{\bf2}}\rangle$ with $2\leqslant  \alpha_{2}\leqslant  2^{n-1}$. Since $R/(R\cap G_{R})$ is isomorphic to a subgroup of $\Aut(G)$ and $\Aut(G)$ is abelian, we have that $\langle a^{\alpha_{\bf2}}, a^{\epsilon}y^{\gamma}\rangle\lhd R$. Let $R_{1}=\langle a^{\alpha_{\bf2}}, a^{\epsilon}y^{\gamma}\rangle$. Since $\langle ax\rangle\cap R_{1}=\bf 1$, we have that $R=R_{1}\rtimes \langle ax\rangle$.

\begin{llemma}\label{ch6lem441}
Suppose that $R$ is regular. Then $R_{1}=\langle a^{\epsilon}y^{\gamma}\rangle$ and $\epsilon_{\bf 2}=2$.
\end{llemma}

\begin{proof} Observe that since $R$ is regular on $G$, we have that $R_{1}$ must be semiregular on $G$. If $\epsilon_{\bf2}\geqslant \alpha_{\bf2}$, then $a^{\epsilon}=(a^{\alpha_{\bf2}})^{\frac{\epsilon_{\bf2}}{\alpha_{\bf2}}\cdot \epsilon_{{\bf2}'}}$, and so we have that $(a^{\alpha_{\bf2}})^{-\frac{\epsilon_{\bf2}}{\alpha_{\bf2}}\cdot \epsilon_{{\bf2}'}}a^{\epsilon}y^{\gamma}=a^{-\epsilon}a^{\epsilon}y^{\gamma}=y^{\gamma}$. Thus $y^{\gamma}\in R_{1}$, which implies that $R_{1}$ is not semiregular as $y^{\gamma}\neq {\bf 1}$ and fixes $\bf1$. Hence we may assume that $\alpha_{\bf2}> \epsilon_{\bf2}$. 

Notice that $\langle a^{\alpha_{\bf2}}, a^{\epsilon}y^{\gamma}\rangle$ is conjugate to $\langle a^{\alpha_{\bf2}}, a^{\epsilon_{\bf2}}y^{\gamma}\rangle$, and so we may assume that $\epsilon_{{\bf2}'}=1$. Since 
$$
a^{\alpha_{\bf2}}=(a^{\epsilon})^{\frac{\alpha_{\bf2}}{\epsilon}},
$$
we have that $a^{\alpha_{\bf2}}=(a^{\epsilon})^{\ell}$ with $\ell=\frac{\alpha_{\bf2}}{\epsilon}$. Thus by Lemma~\ref{ch6lem11},
\begin{equation*}
(a^{\epsilon}y^{\gamma})^{\ell} = a^{M}y^{\gamma\ell} 
\end{equation*}
with $M=\epsilon\frac{1-5^{-\gamma\ell}}{1-5^{-\gamma}}$. By Lemma~\ref{ch6lem01}, we have that $M_{\bf2}=\epsilon_{\bf2}\ell_{\bf2}=\alpha_{\bf2}$. Thus  
\begin{align*}
(a^{\epsilon}y^{\gamma})^{\ell} = & a^{M}y^{\gamma\ell}\\
                                  = & a^{\alpha_{\bf2}M_{{\bf2}'}}y^{\gamma\ell}.
\end{align*}
Suppose that $2^{n-2}\nmid \gamma\ell$. Thus $y^{\gamma\ell}=a^{-\alpha_{\bf2}M_{{\bf2}'}}(a^{\epsilon}y^{\gamma})^{\ell}$, and so $y^{\gamma\ell}\in R_{1}$ which implies that $R_{1}$ is not semiregular. Thus $2^{n-2}\bigm| \gamma\ell$, that is, $(a^{\epsilon}y^{\gamma})^{\ell}=a^{\alpha_{\bf2}M_{{\bf2}'}}$. Since $M_{{\bf2}'}$ is odd, we have that $\langle a^{\alpha_{\bf2}M_{{\bf2}'}}\rangle=\langle a^{\alpha_{\bf2}}\rangle$, which implies that $\langle a^{\alpha_{\bf2}}\rangle=\langle (a^{\epsilon}y^{\gamma})^{\ell}\rangle$. Since $(a^{\epsilon}y^{\gamma})^{\ell}\in \langle a^{\epsilon}y^{\gamma}\rangle$, we have that $\langle a^{\alpha_{\bf2}}\rangle\leqslant  \langle a^{\epsilon}y^{\gamma}\rangle$, which implies that $R_{1}=\langle a^{\epsilon}y^{\gamma}\rangle$.  Since $|R_{1}|=|a^{\epsilon}y^{\gamma}|=2^{n-1}$ and $|a^{\epsilon}y^{\gamma}|=max(\frac{2^{n-2}}{\gamma_{\bf2}}, \frac{2^{n}}{\epsilon_{\bf2}})$ by Lemma~\ref{ch6lem12}, we have that $\epsilon_{\bf2}=2$. \end{proof}

\begin{guess}\label{ch6them441}
If $n\geqslant 3$, then $R$ is regular if and only if one of the following holds
\begin{enumerate}
\item $R=\langle a^{2}y^{2^{n-3}}\rangle\rtimes \langle ax\rangle$ with $R\cap G_R=\langle a^{4}\rangle$, and $R\cong QD_{2^{n}}$ is a quasidihedral group; 
\item $R=\langle a^{2\cdot5^{-1}}y\rangle\times \langle ax\rangle$ and $R\cap G_R=\langle a^{2^{n-1}}\rangle$;
\item $R=\langle a^{2\cdot5^{-1}+2^{n-2}}y\rangle \rtimes \langle ax\rangle$ with $R\cap G_R=\langle a^{2^{n-1}}\rangle$, and $R\cong M_{n}(2)$.
\end{enumerate}
\end{guess}

\begin{proof} {\it (Sufficiency.)} Suppose that $R$ is regular. Then $|R|=2^{n}$. By Lemma~\ref{ch6lem441}, we have that $R_{1}=\langle a^{\epsilon}y^{\gamma}\rangle$ with $\epsilon_{\bf2}=2$. Since $\langle a^{\alpha_{\bf2}}\rangle\leqslant  R_{1}$, we have that $R_{1}\cap G_{R}=R\cap G_{R}$. Let $m=2^{n-2}/\gamma_{\bf 2}$. Notice that for all $u\in R\cap G_{R}$, $u=(a^{\epsilon}y^{\gamma})^{\ell}$ where $2^{n-2}\mid \gamma_{\bf2}\ell_{2}$, and so $R\cap G_R\leqslant \langle (a^{\epsilon}y^{\gamma})^{m}\rangle$. Thus $\langle (a^{\epsilon}y^{\gamma})^{m}\rangle= R\cap G_R$. Since 
\begin{equation*}
(a^{\epsilon}y^{\gamma})^{m}=a^{\epsilon\frac{1-5^{-m\gamma}}{1-5^{-\gamma}}},
\end{equation*}
we have that 
\begin{equation}\label{ch6equ441}
|R\cap G_R|=|(a^{\epsilon}y^{\gamma})^{m}|=\frac{2^{n}}{\epsilon_{\bf2}m_{\bf2}}=2\gamma_{\bf2}.
\end{equation}

Since $R\cap G_R=R_{1}\cap G_R$, and $R/(R\cap G_R)$ is abelian, we have that $[a^{\epsilon}y^{\gamma}, ax]\in R\cap G_R$, which implies that $|[a^{\epsilon}y^{\gamma}, ax]|\bigm| |R\cap G_R|$. Since $(a^{\epsilon}y^{\gamma})^{-1}=a^{-\epsilon\cdot5^{\gamma}}y^{-\gamma}$, we have that
\begin{align}\label{ch6equ442}
[a^{\epsilon}y^{\gamma}, ax] & =  a^{-\epsilon\cdot5^{\gamma}}y^{-\gamma}axa^{\epsilon}y^{\gamma}ax \nonumber \\
                                & = a^{-\epsilon\cdot5^{\gamma}}y^{-\gamma}aa^{-\epsilon}y^{\gamma}xax \nonumber \\
                                & = a^{-\epsilon\cdot5^{\gamma}}a^{(1-\epsilon)5^{\gamma}}a^{-1} \nonumber \\
                                & = a^{-\epsilon\cdot5^{\gamma}+(1-\epsilon)5^{\gamma}-1} \nonumber \\
                                & = a^{(1-2\epsilon)5^{\gamma}-1}.
\end{align}

Recall that
\begin{equation*}
5^{\gamma}=1+\sum\limits_{i=1}^{\gamma}4^{i}{\gamma\choose i},
\end{equation*}
and so 
\begin{align*}
(1-2\epsilon)5^{\gamma}-1 = & (1-2\epsilon)+ (1-2\epsilon)\sum\limits_{i=1}^{\gamma}4^{i}{\gamma\choose i}-1 \\
                              = & -2\epsilon+ (1-2\epsilon)\sum\limits_{i=1}^{\gamma}4^{i}{\gamma\choose i}.
\end{align*}
If $\gamma$ is even, then $(1-2\epsilon)\sum\limits_{i=1}^{\gamma}4^{i}{\gamma\choose i}$ is divisible by 8 and so, since $\epsilon_{\bf2}=2$, we have that $[5^{\gamma}-2\epsilon5^{\gamma}-1]_{\bf2}=4$. Thus $|[a^{\epsilon}y^{\gamma}, ax]|=2^{n-2}$. Thus $2^{n-2}\bigm| 2\gamma_{\bf2}$, and so we have that $\gamma_{\bf2}=2^{n-3}$, which implies that $\gamma=2^{n-3}$. By (\ref{ch6equ441}), we have that $|R\cap G|=2^{n-2}$, that is, $R\cap G=\langle a^{4}\rangle$. Since $\epsilon_{\bf2}=2$, we have that $\epsilon-2=2\cdot\epsilon_{{\bf2}'}-2=2(\epsilon_{{\bf2}'}-1)$ and so $4\mid \epsilon-2$. Thus there exists $\ell$ such that 
$$
a^{\epsilon}y^{2^{n-3}}=(a^{4})^{\ell}\cdot a^{2}y^{2^{n-3}},
$$
which implies that $a^{2}y^{2^{n-3}}\in R_{1}$. Hence $R=\langle a^{2}y^{2^{n-3}}\rangle\rtimes \langle ax\rangle$. Moreover, by Lemma~\ref{ch6lem11}(2) and Lemma~\ref{ch6lem12}, $|a^{2}y^{2^{n-3}}|=|ax|=2^{n-1}$. Since
\begin{align*}
ax(a^{2}y^{2^{n-3}})ax = & a^{-1}y^{2^{n-3}}xax \\
   = & a^{-1}y^{2^{n-3}}a^{-1} \\
   = & a^{-1}a^{-5^{2^{n-3}}}y^{2^{n-3}} \\
   = & a^{-(1+5^{2^{n-3}})}y^{2^{n-3}},
\end{align*}
and by Lemma~\ref{ch6lem00} we have that $1+5^{2^{n-3}}\equiv 2\Mod {2^{n-1}}$, and so 
$$
ax(a^{2}y^{2^{n-3}})ax=a^{-(2+2^{n-1})}y^{2^{n-3}}. 
$$
Notice that by Lemma~\ref{ch6lem01} and Lemma~\ref{ch6lem11}, 
$$
(a^{2}y^{2^{n-3}})^{2^{n-2}}=a^{2\frac{1-5^{-2^{n-3}\cdot 2^{n-2}}}{1-5^{-2^{n-3}}}}=a^{2^{n-1}},
$$
and so
$$
(a^{2}y^{2^{n-3}})^{2^{n-2}-1}=a^{2^{n-1}}(a^{2}y^{2^{n-3}})^{-1}=a^{2^{n-1}}a^{-2\cdot 5^{2^{n-3}}}y^{2^{n-3}}=a^{2^{n-1}-2\cdot 5^{2^{n-3}}}y^{2^{n-3}}.
$$
By Lemma~\ref{ch6lem00} we have that $-2\cdot 5^{2^{n-3}}\equiv -2\Mod {2^{n}}$, and so 
$$
(a^{2}y^{2^{n-3}})^{2^{n-2}-1}=a^{-2+2^{n-1}}y^{2^{n-3}}=a^{-(2+2^{n-1})}y^{2^{n-3}}.
$$
Hence $ax(a^{2}y^{2^{n-3}})ax=(a^{2}y^{2^{n-3}})^{2^{n-2}-1}$. Therefore, $R=\langle a^{2}y^{2^{n-3}}\rangle\rtimes \langle ax\rangle\cong QD_{2^{n}}$, that is a quasidihedral group.

Suppose that  $\gamma$ is odd, that is, $\gamma_{\bf2}=1$. Suppose that  $\gamma\neq 1$. By Lemma~\ref{ch6lem01} we have that $\frac{1-5^{-\gamma}}{1-5^{-1}}$ is odd, so it has an inverse in $Z_{2^{n}}^{*}$. Let $\epsilon'=\epsilon\left(\frac{1-5^{-\gamma}}{1-5^{-1}}\right)^{-1}$. Thus $\epsilon'_{\bf2}=2$, and 
$$
(a^{\epsilon'}y)^{\gamma}=a^{\epsilon}y^{\gamma}.
$$
Hence $R_{1}=\langle a^{\epsilon}y^{\gamma}\rangle\leqslant  \langle a^{\epsilon'}y\rangle$. By Lemma~\ref{ch6lem12}, we have that $|a^{\epsilon'}y|=2^{n-1}$. Since $|R_{1}|=2^{n-1}$, we have that $R_{1}=\langle a^{\epsilon'}y\rangle$. Thus we may assume that $\gamma=1$. 

Since $\gamma=1$, by (\ref{ch6equ441}) we have that $R\cap G_{R}=\langle a^{2^{n-1}}\rangle$, and so $[a^{\epsilon}y, ax]=1$ or $[a^{\epsilon}y, ax]=a^{2^{n-1}}$. First suppose that $[a^{\epsilon}y, ax]=1$. Thus $R$ is abelian, and by (\ref{ch6equ442}) we have that $2^{n}\bigm| [(1-2\epsilon)5-1]$, that is, $2^{n}\bigm| (4-2\epsilon_{\bf2}\epsilon_{{\bf2}'}\cdot5)$, and so $2^{n}\bigm| (4-4\cdot 5\epsilon_{{\bf2}'})$ as $\epsilon_{\bf2}=2$. Thus $2^{n-2}\bigm| (5\epsilon_{{\bf2}'}-1)$, which implies that $\epsilon_{{\bf2}'}=5^{-1}(2^{n-2}\ell+1)$ for some $\ell$. Thus $\epsilon=2\epsilon_{{\bf2}'}=2^{n-1}\cdot5^{-1}\ell+2\cdot5^{-1}$. Since $a^{2^{n-1}\cdot5^{-1}\ell}\in \langle a^{2^{n-1}}\rangle=R\cap G_{R}$, we have that $R=\langle a^{\epsilon}y^{\gamma}\rangle\times \langle ax\rangle=\langle a^{2\cdot5^{-1}}y\rangle\times \langle ax\rangle$. 

Now suppose that  $[a^{\epsilon}y, ax]=a^{2^{n-1}}$, that is, $R$ is nonabelian. By (\ref{ch6equ442}), we have that $[a^{\epsilon}y, ax]=a^{5-2\cdot 5\epsilon-1}=a^{4-2\cdot5\epsilon}$, and so 
\begin{center}
$2^{n-1}\bigm| (4-2\cdot 5\epsilon)$, and $2^{n}\nmid (4-2\cdot 5\epsilon)$.
\end{center}
Thus we have that $2^{n-2}\bigm| (2-5\epsilon)$ and $2^{n-1}\nmid (2-5\epsilon)$, which implies that $\epsilon=2\cdot 5^{-1}+2^{n-2}\ell$ where $\ell=1$ or $\ell=3$. Observe that if $\ell=3$, then $\epsilon=2\cdot 5^{-1}+2^{n-2}+2^{n-1}$. Since $a^{2^{n-1}}\in R\cap G_{R}$, we have that $\langle a^{\epsilon}y\rangle=\langle a^{2^{n-1}}, a^{\epsilon}y\rangle=\langle a^{2\cdot5^{-1}+2^{n-2}}y\rangle$. Hence $R=\langle a^{2\cdot5^{-1}+2^{n-2}}y\rangle\rtimes \langle ax\rangle$.  Moreover, by Lemma~\ref{ch6lem11}(2) and Lemma~\ref{ch6lem12}, $|a^{2\cdot5^{-1}+2^{n-2}}y|=|ax|=2^{n-1}$. Let $t=2\cdot5^{-1}+2^{n-2}$, and so 
$$
ax(a^{t}y)ax= a^{1-t}yxax=a^{1-t}ya^{-1}y^{-1}y=a^{1-t}a^{-5^{-1}}y=a^{1-3\cdot 5^{-1}-2^{n-2}}y.
$$
Notice that by Lemma~\ref{ch6lem11}, 
$$
(a^{t}y)^{2^{n-2}}=a^{t\frac{1-5^{-2^{n-2}}}{1-5^{-1}}}.
$$
Let $M=t\frac{1-5^{-2^{n-2}}}{1-5^{-1}}$, and by Lemma~\ref{ch6lem01} we have that $M_{\bf2}=2^{n-2}t_{\bf2}$. Since $t_{\bf2}=2$, we have that $M_{\bf2}=2^{n-1}$, that is, $(a^{t}y)^{2^{n-2}}=a^{2^{n-1}}$. 
Hence 
$$
(a^{t}y)^{2^{n-2}+1}=a^{2^{n-1}}a^{t}y=a^{2^{n-1}+t}y=a^{2^{n-1}+2\cdot5^{-1}+2^{n-2}}y.
$$
Since 
\begin{align*}
2^{n-1}+2\cdot5^{-1}+2^{n-2} = & (1-1)+2\cdot 5^{-1}+2^{n-2}-2^{n-1} \\
   = & 1-5\cdot 5^{-1}+2\cdot 5^{-1}+2^{n-2}(1-2) \\
   = & 1-3\cdot 5^{-1}-2^{n-2},
\end{align*}
we conclude that $ax(a^{t}y)ax=(a^{t}y)^{2^{n-2}+1}$. Hence in this case, 
$$
R=\langle a^{2\cdot5^{-1}+2^{n-2}}y\rangle\rtimes \langle ax\rangle\cong M_{n}(2).
$$

{\it (Necessity.)} Now suppose that  $R=\langle a^{2}y^{2^{n-3}}\rangle\rtimes \langle ax\rangle$. We show that $R$ is regular by proving that $R$ is transitive on $G$. Since $a^{2}y^{2^{n-3}}$ has order $2^{n-1}$ and is semiregular by Theorem~\ref{ch6them31}, we have that $\langle a^{2}y^{2^{n-3}}\rangle$ has two orbits on $G$, say $O_{1}$ and $O_{2}$. We may assume that $a^{\langle a^{2}y^{2^{n-3}}\rangle}=O_{1}$. Since $(a^{2}y^{2^{n-3}})^{r}=a^{2\cdot\frac{1-5^{-2^{n-3}r}}{1-5^{-2^{n-3}}}}y^{2^{n-3}r}$ and $2\cdot\frac{1-5^{-2^{n-3}r}}{1-5^{-2^{n-3}}}$ is even for all $1\leqslant  r\leqslant  2^{n-1}$, we have that $O_{1}=\{a^{\ell}| \ell \ \text{is odd}\}$. Since $a^{ax}=a^{-2}$, we have that $a^{ax}\in O_{2}$. Since $\langle a^{2}y^{2^{n-3}}\rangle$ is transitive on $O_{2}$, we conclude that $R$ is transitive on $G$. Hence $R$ is regular. 

With the same arguments, we conclude that if $R=\langle a^{2\cdot5^{-1}}y\rangle\times \langle ax\rangle$, then $R$ is regular; and if $R=\langle a^{2\cdot5^{-1}+2^{n-2}}y\rangle\rtimes \langle ax\rangle$, then $R$ is regular as well. \end{proof}

\bigskip

{\it\bf Proof of Theorem~\ref{ch6maintheo}}. Let $R$ be a regular subgroup of $H$. Then $R$ belongs to one of  the five types $(1)-(5)$ in Theorem~\ref{ch6them41}. If $R=R\cap G_{R}$, then $R=G_R
$ as $R$ is regular. Thus we may assume that $R\cap G_{R}<R$. 

Suppose that $R=(R\cap G_{R})\langle  a^{\epsilon}y^{\gamma}\rangle$, that is, $R=\langle a^{2^{t}},  a^{\epsilon}y^{\gamma}\rangle$ where $1\leqslant  t\leqslant  n-1$. Thus by Lemma~\ref{ch6lem411}, we have that if $R$ is regular then $R$ is conjugate to $\langle  ay^{\gamma}\rangle$. Further by Theorem~\ref{ch6them411} we have that $R$ is regular if and only if $R=\langle  ay^{2^{t}}\rangle$ where $0\leqslant  t\leqslant  n-2$. 

Suppose that $R=(R\cap G_{R})\langle  a^{\epsilon}x\rangle$, that is, $R=\langle a^{2^{t}}, a^{\epsilon}x\rangle$ where $1\leqslant  t\leqslant  n-1$. By Theorem~\ref{ch6them421}, we have that $R$ is regular if and only if $R=\langle a^{2}, ax\rangle$. 

Suppose that $R=(R\cap G_{R})\langle  a^{\epsilon}xy^{\gamma}\rangle$, that is, $R=\langle a^{2^{t}}, a^{\epsilon}xy^{\gamma}\rangle$ where $1\leqslant  t\leqslant  n-1$. Then by Theorem~\ref{ch6them431}, $R$ is regular if and only if $R=\langle a^{2}, axy^{2^{n-3}}\rangle$.

Suppose that $R=(R\cap G_{R})\langle  a^{\epsilon}x, a^{\epsilon'}y^{\gamma}\rangle$. Since $R$ is regular, by Theorem~\ref{ch6them31} we have that $\epsilon$ must be odd. Up to conjugacy, we may assume that $R=(R\cap G_{R})\langle  ax, a^{\epsilon'}y^{\gamma}\rangle$. By Lemma~\ref{ch6lem441}, if $R$ is regular, then $\epsilon'$ must be even. By Theorem~\ref{ch6them441}, if $\gamma=2^{n-3}$, then $R=\langle a^{2}y^{2^{n-3}}\rangle\rtimes \langle ax\rangle$. If $\gamma=1$, then either $R=\langle a^{2\cdot 5^{-1}}y\rangle\times \langle ax\rangle$ that is abelian, or $R=\langle a^{2\cdot 5^{-1}+2^{n-2}}y\rangle \rtimes \langle ax\rangle$ that is nonabelian. \qed

\subsection{Normality of the Cyclic Regular Subgroups in $\Hol(G)$}

Recall that $G$ has the NNN-property if and only if there exists an NNN-graph for it. Suppose that $\Gamma$ is  a normal circulant for $G$.  All regular subgroups of $\Aut(\Gamma)$ are regular in $\Hol(G)$ as $\Aut(\Gamma)\leqslant \Hol(G)$. If all the regular cyclic subgroups of $\Aut(\Gamma)$ are normal in $\Hol(G)$, then we have that $\Gamma$ is not an NNN-graph for $G$. If there exists a non-normal regular cyclic subgroup $H$ in $\Hol(G)$, then next we check whether $H$ is a non-normal subgroup of $\Aut(\Gamma)$. If it is, then we have found an NNN-graph for $G$, which implies that $G$ has the NNN-property. Hence it is necessary to investigate the normality of the regular cyclic subgroups of $\Hol(G)$, which are the regular subgroups of Type (2) in Theorem~\ref{ch6maintheo}.

\begin{guess}\label{ch6them51}
Let $R$ be a regular cyclic subgroup of $\Hol(G)$. Then $R$ is normal in $\Hol(G)$ if and only if  $R=G_R$ or $R=\langle ay^{2^{n-3}}\rangle$. 
\end{guess}

\begin{proof} Since $R$ is cyclic,
% in $\Hol(G)$, 
by Theorem~\ref{ch6maintheo}, we have that $R=G_R$ or $R=\langle ay^{2^{t}}\rangle$ for $0\leq t\leq n-3$. Clearly if $R=G_R$, then $R$ is normal in $\Hol(G)$. 

Suppose that $R\neq G_{R}$. Let $N=R\cap G_{R}$. Since $R=\langle ay^{2^{t}}\rangle$ and $(ay^{2^{t}})^{2^{n-2-t}}\in G_R$, by Lemma~\ref{ch6lem11} we have that 
\begin{align*}
(ay^{2^{t}})^{2^{n-2-t}} & = a^{\frac{1-5^{-2^{n-2-t}\cdot 2^{t}}}{1-5^{-2^{t}}}},                                     
\end{align*}
and so by Lemma~\ref{ch6lem01} we have that $R\cap G_{R}=\langle a^{2^{n-2-t}}\rangle$. Thus $|R\cap G_{R}|=2^{t+2}$. Now consider $(ay^{2^{t}})^{x}$, that is,
\begin{equation*}
(ay^{2^{t}})^{x}=a^{-1}y^{2^{t}}.
\end{equation*}
Thus 
\begin{align*}
((ay^{2^{t}})^{x})^{-1}ay^{2^{t}} & = y^{-2^{t}}a^{2}y^{2^{t}}\\
                                                   &  = (a^{2})^{y^{2^{t}}}\\
                                                   & = a^{2\cdot 5^{2^{t}}},                                          
\end{align*}
and so $((ay^{2^{t}})^{x})^{-1}ay^{2^{t}}\in \langle a^{2}\rangle$. Notice that if $R\unlhd \Hol(G)$, then $((ay^{2^{t}})^{x})^{-1}ay^{2^{t}}\in R$. Thus $\langle a^{2}\rangle\leqslant R\cap G_{R}$, and so $|R\cap G_{R}|\geqslant 2^{n-1}$. Since $|R\cap G_{R}|=2^{t+2}$, we have that $t= n-3$. Hence $R$ is non-normal in $\Hol(G)$ for all $0\leqslant t\leqslant n-4$. 

Now let $R=\langle ay^{2^{n-3}}\rangle$. Thus $R\cap G_{R}=\langle a^{2}\rangle$, which implies that 
\begin{equation*}
\Hol(G)/(R\cap G_{R})\cong C_{2}\times \langle x, y\rangle.
\end{equation*}
Since the quotient group is abelian, we have that $R\unlhd \Hol(G)$. \end{proof}

\begin{coro}\label{ch6coro51}
Let $\Gamma$ be a circulant for $Z_{2^{n}}$. If $\Gamma$ is an NNN-graph for $Z_{2^{n}}$, then $y^{2^{n-4}}\in \Aut(\Gamma)$. 
\end{coro}
\begin{proof}
If $\Gamma$ is NNN for $Z_{2^{n}}$, then $\Aut(\Gamma)$ contains two isomorphic regular subgroups where one is normal and the other one is non-normal. Let $R$ be a non-normal one. By Theorem~\ref{ch6maintheo} we have that $R=\langle ay^{2^{t}}\rangle$ for some $0\leqslant t\leqslant n-3$. If $y^{2^{n-4}}\notin \Aut(\Gamma)$, then $ay^{2^{n-4}}\notin \Aut(\Gamma)$, which implies that $R=\langle ay^{2^{n-3}}\rangle$. Thus by Lemma~\ref{ch6them51} we have that $R\unlhd \Aut(\Gamma)$, which leads to a contradiction. 
\end{proof}

Theorem~\ref{ch6them51} leads to the following natural question: does there exist a circulant for $G$ whose automorphism group equals the holomorph group of $G$? If the answer is yes, then such a circulant is an NNN-graph for $G$. Before showing that the answer is no, we need the following lemma.

\begin{llemma}\label{lexicononnormal}%\cite[Lemma 5.3.4]{yxuthesis2019}
Let $\Gamma$ be a circulant for $Z_{2^n}$ with $n\geqslant3$. Suppose that $\Gamma=X[Y]$ (lexicographic product) where $X$ and $Y$ are nontrivial circulants. Then $\Gamma$ is non-normal for $Z_{2^n}$.
\end{llemma}

\begin{proof}
Let $X=\Cay(Z_{2^{n-t}}, S)$ and $Y=\Cay(Z_{2^{t}}, T)$. Thus  $\Aut(Y) \, \mathrm{wr} \, \Aut(X)\leqslant \Aut(\Gamma)$, and so 
$$
|\Aut(X)|\cdot |\Aut(Y)|^{|X|} \ \textrm{divides} \ |\Aut(\Gamma)|.
$$
Notice that a circulant has a regular automorphism group if and only if it is isomorphic to $K_{2}$, and so $X$ and $Y$ cannot be both such circulants as $n\geqslant 3$. Thus we have that 
\begin{equation}\label{ch5equlexico}
|\Aut(\Gamma)|\geqslant|\Aut(X)|\cdot |\Aut(Y)|^{|X|}>2^{n-t}(2^{t})^{2^{n-t}}=2^{2^{n-t}\cdot t+n-t}.
\end{equation}
Recall that $\Hol(Z_{2^{n}})=G_{R}\rtimes \Aut(G)$, and $|\Hol(Z_{2^{n}})|=2^{2n-1}$. Let $f(t)=2^{n-t}\cdot t+n-t-(2n-1)=2^{n-t}\cdot t-n-t+1$. Thus $f'(t)=-\left(\ln\left(2\right)t-1\right){\cdot}2^{n-t}-1$. Notice that when $2\leqslant t\leqslant n-1$, $\ln\left(2\right)t-1>0$, and so $f'(t)<0$. Since $f(n-1)=0$ and $f(1)=2^{n-1}-n$, we have that $f(t)\geqslant \min(0, 2^{n-1}-n)$. Since $2^{n-1}\geqslant n$ for $n\in \mathbb{N}^{+}$, we have that $f(t)\geqslant 0$. Thus $2^{n-t}\cdot t+n-t\geqslant 2n-1$. Hence by (\ref{ch5equlexico}) we have that 
$$
|\Aut(\Gamma)|\geqslant|\Aut(X)|\cdot |\Aut(Y)|^{|X|}>2^{n-t}(2^{t})^{2^{n-t}}=2^{2^{n-t}\cdot t+n-t}\geqslant 2^{2n-1},
$$
that is, $|\Aut(\Gamma)|>|\Hol(Z_{2^{n}})|$. Hence $\Gamma$ is non-normal for $Z_{2^{n}}$ when $n\geqslant 3$.
\end{proof}

\begin{llemma}\label{ch6lem51}
Let $\Gamma$ be a circulant for $G$. If $y\in \Aut(\Gamma)$, then $\Gamma$ is not normal for $G$. In particular, $\Gamma$ is not an NNN-graph for $G$.
\end{llemma} 

\begin{proof} First notice that  $x\in \Aut(\Gamma)$. Let $S$ be the connection set of $\Gamma$, and $S_{0}$ be the set of generators of $G$ in $S$. Obviously there exists some $a^{\epsilon}\in S_{0}$ where $\epsilon$ is odd. Since $y\in \Aut(\Gamma)$, we  have that $y\in \Aut(\Gamma)_{\bf1}$, and so $(a^{\epsilon})^{\langle x, y\rangle}\subset S_{0}$. Note that $(a^{\epsilon})^{\langle x, y\rangle}$ is the set of all generators of $G$. Let $H=\langle a^{2}\rangle$. Thus $(a^{\epsilon})^{\langle x, y\rangle}=S_{0}=S\backslash H=aH$, that is, $H<_{S}G$. Thus by \cite[Theorem 1.2]{kovacs2012cayley} we have that $\Gamma$ is a lexicographic product, and so it follows from Lemma~\ref{lexicononnormal} that $\Gamma$ is not an NNN-graph for $G$. \end{proof}

\section{The NNN-Property of Cyclic Groups}

The purpose of this section is to prove Theorem~\ref{main0}. By Theorem~\ref{notdivisibleby8}, it is left to check the NNN-property of cyclic groups whose order is divisible by 8. Let $G=Z_{n}$ where $n=2^{k_{1}}p_{2}^{k_{2}}\cdots p_{t}^{k_{t}}$ with $k_{1}\geqslant 3$ and $p_{i}$ an odd prime. 
Recall that 
\begin{equation}\label{equ1}
G\cong Z_{2^{k_{1}}}\times Z_{p_2^{k_2}}\times\cdots\times Z_{p_{t}^{k_{t}}},
\end{equation}
and 
\begin{equation}\label{equ2}
\mathrm{Aut} (G)=\Aut(Z_{2^{k_{1}}})\times   \Aut(Z_{p_{2}^{k_{2}}}) \times\cdots\times \Aut(Z_{p_{t}^{k_{t}}}).
\end{equation}
Let $a=(a_{1}, \ldots, a_{t})$ be a generator of $G$ where $a_{i}$ generates $Z_{p_{i}^{k_{i}}}$, and for $g=(g_{1}, \ldots, g_{t})\in G$ and $\pi=(\pi_{1}, \ldots, \pi_{t})\in \mathrm{Aut} (G)$ we have that, 
\begin{equation*}%\label{euq3}
g^{\pi}=(g_{1}^{\pi_{1}}, \ldots, g_{t}^{\pi_{t}}). 
\end{equation*}
Let $\Hol(G)=G\rtimes \Aut(G)$. By (\ref{equ1}) and (\ref{equ2}) we have that 
\begin{equation}\label{equ4}
\Hol(G)=\Hol(Z_{2^{k_{1}}})\times\Hol(Z_{p_{2}^{k_{2}}})\times \cdots \times \Hol(Z_{p_{t}^{k_{t}}}). 
\end{equation}

Let $\Gamma=\mathrm{Cay}(G, S)$ be a normal circulant and $A=\mathrm{Aut}(\Gamma)$. Let $R\leqslant A$ be an cyclic regular subgroup where $R\neq G_R$ and $R\cong G$. By Lemma~\ref{ch5them31} we have that 
\begin{equation}\label{equ5}
N=Z_{2^{\ell}}\times \prod_{i=1}^{t}Z_{p_{i}^{k_{i}}}
\end{equation}
for some $1\leqslant \ell\leqslant k_{1}-1$.

Let $h\in R$. Since $R\leqslant \Hol(G)$, we can write $h=g\pi$ where $g\in G_R$ and $\pi\in \Aut(G)$. Let $B=\prod_{i=1}^{t}Z_{p_{i}^{k_{i}}}\leqslant G_R\cap R$.

\begin{llemma}\label{lem1}
$R\leqslant \Hol(Z_{2^{k_{1}}})\times B$. 
\end{llemma}

\begin{proof} Notice that for all $2\leqslant i\neq j\leqslant t$, we have that $C_{\Aut(Z_{p_{i}^{k_{i}}})}(Z_{p_{i}^{k_{i}}})=\{\bf 1\}$ and $C_{\Aut(Z_{p_{i}^{k_{i}}})}(Z_{p_{j}^{k_{j}}})=\Aut(Z_{p_{i}^{k_{i}}})$. Let $h\in R$. Since $R$ is abelian and $B\leqslant R$, we have that $h=g\pi$ where $\pi\in \Aut(Z_{2^{k_{1}}})$, which implies that $R\leqslant G_R\rtimes \Aut(Z_{2^{k_{1}}})$. This completes the proof as $G_R\rtimes \Aut(Z_{2^{k_{1}}})=\Hol(Z_{2^{k_{1}}})\times B$.  \end{proof}

%\begin{llemma}\label{lem2}
%Let $h\in \Hol(Z_{2^{k_{1}}})\times B$. Then $h$ is semiregular if and only if $h=h_{1}b$ with $h_{1}$ semiregular in $\Hol(Z_{2^{k_{1}}})$ and $b\in B$.
%\end{llemma}
%
%\cmt{this is false. eg $h=(12)(345678)$ of order 6. Remove } 
%
%\begin{proof} If $h_{1}=\bf1$, then $h$ is semiregular as $B$ is semiregular on $G$. Thus we may assume that $h_{1}\neq \bf 1$. Notice that $|h|=|h_{1}|\cdot|b|$. Since $h^{\frac{|h|}{2}}$ is the unique involution of $\langle h\rangle$, we have that $h$ is semiregular if and only if $h^{\frac{|h|}{2}}$ is semiregular. Since $$
%h^{\frac{|h|}{2}}=h_{1}^{\frac{|h_{1}|}{2}\cdot |b|}, 
%$$
%we have that $h$ is semiregular if and only if $h_{1}$ is semiregular. \end{proof}

Recall that $Z_{2^{k_{1}}}=\langle a_{1}\rangle$ and $\Aut(Z_{2^{k_{1}}})=\langle x\rangle\times\langle y\rangle$ where $x: a_{1}\mapsto a_{1}^{-1}$ and $y: a_{1}\mapsto a_{1}^{5}$. Note that $B$ is generated by $(a_{2}, \ldots, a_{t})$. By the Second Isomorphism Theorem and Lemma~\ref{lem1}, we have that 
\begin{equation}\label{equ6}
R/(R\cap G_R)\cong RG_R/G_R\leqslant \langle x\rangle\times \langle y\rangle.
\end{equation}
Recall that $N=R\cap G_R=Z_{2^{\ell}}\times B$ for some $1\leqslant \ell\leqslant k_{1}-1$. Thus  $R=\langle N, h\rangle$  for some $h\in R\backslash G_R$. By Lemma~\ref{lem1} we have that $h=rb$ with $r\in \Hol(Z_{2^{k_{1}}})$ and $b\in B$. Since $B\leqslant R$, we have that $R=K\times B$ where $K=\langle Z_{2^{\ell}}, r\rangle\leqslant \Hol(Z_{2^{k_1}})$.

Now $G=Z_{2^{k_1}}\times B$ and $\Hol(Z_{2^{k_1}})\leqslant \Hol(G)$ fixes $Z_{2^{k_1}}$ setwise. Thus as $R$ acts regularly on $G$, it follows that $K$ acts regularly on $Z_{2^{k_1}}$.
Since $K$ is cyclic, it follows from Theorem~\ref{ch6maintheo} that up to conjugacy $R=\langle ay^{2^{\ell}}\rangle\times B$ with $0\leqslant \ell\leqslant k_{1}-3$. Note that if $k_{1}=3$, then it follows from Theorem~\ref{ch6them51} that $R\unlhd \Hol(G)$. Thus from here we may assume that $k_{1}\geqslant 4$.

\begin{llemma}\label{lem3}
If $\Gamma$ is non-normal for $R$, then $y^{2^{k_{1}-4}}\in \Aut(\Gamma)$.
\end{llemma}

\begin{proof} By Corollary~\ref{ch6coro51} and (\ref{equ4}), it implies that if $R$ is non-normal in $A$, then $y^{2^{k_{1}-4}}\in \Aut(\Gamma)$.  \end{proof}

\begin{guess}\label{them1}
If $y^{2^{k_{1}-4}}\in \Aut(\Gamma)$, then $\Gamma$ is non-normal for $G$.
\end{guess}

\begin{proof} Recall that $G_R=Z_{2^{k_{1}}}\times B$ and $\Aut(G)=\Aut(Z_{2^{k_{1}}})\times \Aut(B)$. In  $Z_{2^{k_{1}}}$, for each $0\leqslant i\leqslant k_{1}-1$, let $H_{i}=\langle a_{1}^{2^{i}}\rangle$ and $M_{i}$ be the set of generators of $H_{i}$. Let 
\begin{equation*}
S_{i}=(M_{i}\times B)\cap S.
\end{equation*} 
Thus $S=\bigcup_{i=0}^{k_{1}-1}S_{i}$ is a disjoint union.

Suppose to the contrary that $\Gamma$ is normal for $G$, that is, $A=\Aut(\Gamma)=G_R\rtimes \Aut(G, S)$. Let $s\in S$, and write $s=(t_{1}, t_{2})$ with $t_{1}\in Z_{2^{k_{1}}}$ and $t_{2}\in B$. Let $f\in \Aut(B)$ where $f: (a_{2}, \ldots, a_{t})\mapsto (a_{2}^{-1}, \ldots, a_{t}^{-1})$. Since $S^{-1}=S$, we have that $xf\in \Aut(G, S)$. Since $y^{2^{k_{1}-4}}\in A$, we have that $y^{2^{k_{1}-4}}\in \Aut(G,S)$, and so $y^{2^{k_{1}-3}}\in \Aut(G, S)$. Let $\rho=y^{2^{k_{1}-3}}$. We have that $\rho$ acts trivially on $B$ and 
$$
\rho: a_{1}\mapsto a_{1}^{2^{k_{1}-1}+1}.
$$ 
Let $\sigma=x\rho f$, and so 
$$
\sigma: (a_{1}, a_{2},\ldots, a_{t})\mapsto (a_{1}^{2^{k_{1}}-1}, a_{2}^{-1}, \ldots, a_{t}^{-1}).
$$
Obviously $\rho, \sigma\in \Aut(G, S)$.

Since $\Gamma$ is connected, we have that $S_{0}\neq \emptyset$. Let $s=(t_{1}, t_{2})\in S_{0}$, and so $t_{1}=a_{1}^{2m+1}$ for some odd $m$. Since 
\begin{equation}\label{equrho}
t_{1}^{\rho}=a_{1}^{(2m+1)(2^{k_{1}-1}+1)}=a_{1}^{2m+1}a_{1}^{2^{k_{1}-1}}=t_{1} a_{1}^{2^{k_{1}-1}}
\end{equation}
and 
\begin{equation}\label{equsigma}
t_{1}^{x\rho}=(t_{1}^{-1})^{\rho}=(t_{1}^{\rho})^{-1}=(t_{1} a_{1}^{2^{k_{1}-1}})^{-1}=t_{1}^{-1}a_{1}^{2^{k_{1}-1}},
\end{equation}
we have that 
\begin{equation*}
s^{\rho}=(t_{1}^{\rho}, t_{2})=(t_{1}, t_{2})(a_{1}^{2^{k_{1}-1}}, {\bf1})=s a_{1}^{2^{k_{1}-1}},
\end{equation*}
and 
\begin{equation}\label{equ7}
s^{\sigma}=(t_{1}^{x\rho}, t_{2}^{f})=(t_{1}^{-1}, t_{2}^{-1})(a_{1}^{2^{k_{1}-1}}, {\bf1})=s^{-1} a_{1}^{2^{k_{1}-1}}.
\end{equation}
Thus
\begin{equation}\label{equ8}
S_{0}=\bigcup_{s\in I}\{s, s^{-1}, s a_{1}^{2^{k_{1}-1}}, s^{-1} a_{1}^{2^{k_{1}-1}}\}
\end{equation}
where $I\subseteq M_{0}\times B$. 

Let $s=(a_{1}^{2^{i}m}, t_{2})$ with $1\leqslant i\leqslant k_{1}-1$ and $m$ odd. Since
\begin{equation}\label{ch5equrho}
t_{1}^{\rho}=a_{1}^{2^{i}m(2^{k_{1}-1}+1)}=a_{1}^{2^{i}m}=t_{1},
\end{equation}
it follows from (\ref{ch5equrho}) that 
\begin{equation}\label{equ9}
s^{\rho}=(a_{1}^{2^{i}m}, t_{2})=s,
\end{equation}
and 
\begin{equation}\label{equ10}
s^{\sigma}=s^{x\rho\cdot f}=s^{xf}=s^{-1}.
\end{equation}

Since $y^{2^{k_{1}-4}}, xf\in \Aut(G, S)$, we have that $s^{\langle y^{2^{k_{1}-4}}, xf\rangle}\subset S$ for all $s\in S_{1}$. Clearly $\rho, \sigma\in \langle y^{2^{k_{1}-4}}, xf\rangle$. Let $s\in S_{1}$, and write $s=(t_{1}, t_{2})$ where $t_{2}\in B$ and $t_{1}=a_{1}^{2m}$ with $m$ odd. By Lemma~\ref{ch6lem00} we have the following:
\begin{equation}\label{ch6equ61}
t_{1}^{y^{2^{k_{1}-4}}}=a_{1}^{2m(5^{2^{k_{1}-4}})}=a_{1}^{2m(2^{k_{1}-2}\ell+1)}=a_{1}^{2^{k_{1}-1}+2m}=t_{1} a_{1}^{2^{k_{1}-1}} \textrm{for some $\ell\neq 0$},
\end{equation}
\begin{equation}\label{ch6equ62}
t_{1}^{(y^{2^{k_{1}-4}})^{3}}=t_{1}^{y^{3\cdot2^{k_{1}-4}}}=a_{1}^{2m(5^{3\cdot2^{k_{1}-4}})}=a_{1}^{2m(2^{k_{1}-2}\ell'+1)}=a_{1}^{2^{k_{1}-1}+2m}=t_{1} a_{1}^{2^{k_{1}-1}} \textrm{for some $\ell'\neq 0$}.
\end{equation}
Thus we have that 
\begin{equation}\label{neweq}
s^{y^{2^{k_{1}-4}}}=(t_{1}^{y^{2^{k_{1}-4}}}, t_{2})=(t_{1} a_{1}^{2^{k_{1}-1}}, t_{2})=s a_{1}^{2^{k_{1}-1}},
\end{equation}
and 
\begin{equation*}
s^{(y^{2^{k_{1}-4}})^{3}}=(t_{1}^{(y^{2^{k_{1}-4}})^{3}}, t_{2})=(t_{1}a_{1}^{2^{k_{1}-1}}, t_{2})=s a_{1}^{2^{k_{1}-1}}. 
\end{equation*}
Further we have that 
\begin{equation*}
s^{xy^{2^{k_{1}-4}}f}=s^{-1} a_{1}^{2^{k_{1}-1}}  \quad \textrm{ and } \quad s^{x(y^{2^{k_{1}-4}})^{3}f}=s^{-1} a_{1}^{2^{k_{1}-1}},
\end{equation*}
which implies that 
\begin{equation*}
s^{\langle y^{2^{k_{1}-4}}, xf\rangle}=\{s, s^{-1}, s a_{1}^{2^{k_{1}-1}}, s^{-1} a_{1}^{2^{k_{1}-1}}\}. 
\end{equation*}
If $y^{2^{t}}\in \Aut(G, S)$ with $t\leqslant k_{1}-4$, then 
\begin{equation*}
s^{y^{2^{t}}}=(a_{1}^{2m\cdot 5^{2^{t}}}, t_{2})\in S_{1}.
\end{equation*}
Thus if $y^{2^{k_{1}-4}}\in A$, then 
\begin{equation}\label{equ11}
S_{1}=\bigcup_{s\in S_{1}}\{s, s^{-1}, s a_{1}^{2^{k_{1}-1}}, s^{-1} a_{1}^{2^{k_{1}-1}}\}.
\end{equation}

First suppose that $S=S_{0}\cup S_{1}$. Let $X=\{{\bf1}, a_{1}^{2^{k_{1}-1}}\}$. It follows from (\ref{equ8}) and (\ref{equ11}) that $S$ is a union of $X$-cosets, that is $X <_{S}G$, and so by \cite[Theorem 1.2]{kovacs2012cayley}  that $\Gamma$ is a lexicographic product of nontrivial circulants, which by Lemma~\ref{lexicononnormal} it implies that $\Gamma$ is non-normal for $G$.
 
Now suppose that $S=S_{0}\cup S_{1}\bigcup(\cup_{i\in I}S_{i})$ with $I\subseteq \{2, \ldots, k_{1}-1\}$ and $I\neq \emptyset$. Note that $a_{1}^{2^{k_{1}-1}}$ is the unique involution in $G$, and so $a_{1}^{2^{k_{1}-1}}\in \langle a_{1}^{4}\rangle\times B$. Hence, if $g\in \langle a_{1}^{4}\rangle\times B$, then $g a_{1}^{2^{k_{1}-1}}\in \langle a_{1}^{4}\rangle\times B$. We define a map $\theta: G\mapsto G$ as below: for all $g\in G$,
\begin{equation*}
g^{\theta}= \begin{cases}
               g               & g\notin (\langle a_{1}^{4}\rangle\times B) a_{1}^{2}, \\
               g a_{1}^{2^{k_{1}-1}}               & g\in (\langle a_{1}^{4}\rangle\times B) a_{1}^{2}
               \end{cases}
\end{equation*}
Clearly $\theta$ is a bijection. Also $\theta$ acts trivially on $\langle a_{1}^{4}\rangle\times B$, $(\langle a_{1}^{4}\rangle\times B)a_{1}$ and $(\langle a_{1}^{4}\rangle\times B)a_{1}^{3}$ respectively, and by (\ref{neweq}) acts on $(\langle a_{1}^{4}\rangle\times B) a_{1}^{2}$ as $y^{2^{k_{1}-4}}$, as every element in $(\langle a_{1}^{4}\rangle\times B) a_{1}^{2}$ is of the form $(t_1,t_2)$ with $t_1=a_1^{2m}$ with $m$ odd.

Let $\{u, v\}\in E(\Gamma)$, and $s=uv^{-1}$. If $u, v$ are fixed by $\theta$, then $u^{\theta}(v^{\theta})^{-1}=uv^{-1}=s\in S$. Similarly if both $u$ and $v$ are not fixed by $\theta$, then $u^{\theta}(v^{\theta})^{-1}=uv^{-1}=s\in S$. Thus we may assume that only one of $u$ and $v$, say $u$, is fixed by $\theta$. This implies that $s\in S_{0}\bigcup S_{1}$. Then 
\begin{equation*}
u^{\theta}(v^{\theta})^{-1}=uv^{-1}a_{1}^{2^{k_{1}-1}} =s a_{1}^{2^{k_{1}-1}},
\end{equation*}
and so by (\ref{equ8}) and (\ref{equ11}) we have that $\{u^{\theta}, v^{\theta}\}\in E(\Gamma)$. Hence $\theta\in A$. 

Since ${\bf1}\notin (\langle a_{1}^{4}\rangle\times B) a_{1}^{2}$, we have that $\theta$ fixes ${\bf1}$. Clearly $\theta$ fixes the generators of $G$, and so $\theta\notin \Aut(G, S)$ as $\theta\neq {\bf1}$. Hence $\Gamma$ is non-normal for $G$.  \end{proof}

{\it Proof of Theorem~\ref{main0}.} Let $G=Z_{n}$, and $\Gamma$ be a normal circulant for $G$. By Lemma~\ref{lem3} and Theorem~\ref{them1} we have that $G$ does not have the NNN-property if $8$ divides $n$. Thus by Theorem~\ref{notdivisibleby8}, we complete the proof. \qed

\subsection*{Acknowledgements}
Many of the results of this paper appeared in the last author's PhD thesis, which was supported by an Australian Government Research Training Program (RTP) Scholarship and the University of Western Australia Safety-Net-Top-Up scholarship. This research also forms part of the ARC Discovery Project DP150101066. The second author was supported by the Australian Research Council grant DE160100081.

\bibliographystyle{acm}%Used BibTeX style is unsrt
\bibliography{ref}

\end{document}